\documentclass{article}

\usepackage[final]{neurips_2021}

\usepackage[utf8]{inputenc} 
\usepackage[T1]{fontenc}    
\usepackage{hyperref}       
\usepackage{url}            
\usepackage{booktabs}       
\usepackage{amsfonts}       
\usepackage{nicefrac}       
\usepackage{microtype}      
\usepackage{xcolor}         

\usepackage{amsthm, amsmath, amssymb, natbib, graphicx, url}
\usepackage[algo2e,ruled]{algorithm2e}
\usepackage{times}
\usepackage{nicefrac}
\usepackage{enumitem} 
\usepackage{tikz}
\usepackage[capitalise]{cleveref}
\usepackage{thmtools,thm-restate} 

\newcommand{\cF}{\mathcal{F}}

\newcommand{\cL}{\mathcal{L}}
\newcommand{\cM}{\mathcal{M}}
\newcommand{\cN}{\mathcal{N}}

\newcommand{\cX}{\mathcal{X}}

\newcommand{\I}{\mathbb{I}}

\newcommand{\N}{\mathbb{N}}
\newcommand{\Ns}{\mathbb{N}^*}

\newcommand{\R}{\mathbb{R}}
\newcommand{\Rd}{\R^d}

\newcommand{\ba}{\boldsymbol{a}}

\newcommand{\bx}{\boldsymbol{x}}
\newcommand{\bxs}{\boldsymbol{x}^\star}
\newcommand{\by}{\boldsymbol{y}}

\newcommand{\bzero}{\boldsymbol{0}}

\DeclareMathOperator*{\argmax}{argmax}

\newcommand{\dif}{\,\mathrm{d}}
\DeclareMathOperator*{\card}{card}
\newcommand{\e}{\varepsilon}

\renewcommand{\epsilon}{\e}
\newcommand{\lrb}[1]{\left(#1\right)}
\newcommand{\brb}[1]{\bigl(#1\bigr)}
\newcommand{\Brb}[1]{\Bigl(#1\Bigr)}

\newcommand{\bsb}[1]{\bigl[#1\bigr]}

\newcommand{\bcb}[1]{\bigl\{#1\bigr\}}

\newcommand{\bce}[1]{\bigl\lceil#1\bigr\rceil}

\newcommand{\lfl}[1]{\left\lfloor#1\right\rfloor}

\newcommand{\labs}[1]{\left\lvert#1\right\rvert}
\newcommand{\babs}[1]{\bigl\lvert#1\bigr\rvert}

\newcommand{\lno}[1]{\left\lVert#1\right\rVert}

\newcommand{\lip}{Lipschitz}
\newcommand{\leb}{Lebesgue}
\newcommand{\mink}{Minkowski}
\newcommand{\piya}{Piyavskii}

\newcommand{\vol}{\mathrm{vol}}
\newcommand{\representative}{representative}

\newcommand{\s}{\subset}

\newcommand{\iop}{\infty}

\newcommand{\eo}{\e_0}
\newcommand{\me}{m_\e}

\newcommand{\fracc}[2]{#1 / #2}
\newcommand{\SNC}{S_{\mathrm{NC}}}
\newcommand{\SC}{S_{\mathrm{C}}}
\newcommand{\Lip}{\mathrm{Lip}}
\newcommand{\err}{\mathrm{err}}
\newcommand{\suppl}{Supplementary Material}
\newcommand{\cdoo}{\mathrm{c.DOO}}
\newcommand{\ncdoo}{\mathrm{nc.DOO}}
\newcommand{\papertitle}{Instance-Dependent Bounds for Zeroth-order \lip{} Optimization with Error Certificates}

\newtheorem{assumption}{Assumption}
    \crefname{assumption}{Assumption}{Assumptions}
\newtheorem{definition}{Definition}
    \crefname{definition}{Definition}{Definitions}
\newtheorem{theorem}{Theorem}
\newtheorem{lemma}{Lemma}
\newtheorem{proposition}{Proposition}

\newtheorem*{prop*}{\Cref{prop:counter:example:Nepsilon}}
\newtheorem{remark}{Remark}
\newtheorem*{theorem*}{Theorem}

\title{\papertitle}

\author{%
  Fran\c{c}ois Bachoc\\
 Institut de Math\'ematiques de Toulouse \& University Paul Sabatier \\
  \texttt{francois.bachoc@math.univ-toulouse.fr} \\
  \And
   Tommaso Cesari \\
   Toulouse School of Economics \\
   \texttt{tommaso-renato.cesari@univ-toulouse.fr} \\
   \AND
   S\'ebastien Gerchinovitz \\
   IRT Saint Exup\'ery \& Institut de Math\'ematiques de Toulouse \\
   \texttt{sebastien.gerchinovitz@irt-saintexupery.com} \\
}

\begin{document}

\maketitle

\begin{abstract}
We study the problem of zeroth-order (black-box) optimization of a \lip{} function $f$ defined on a compact subset $\cX$ of $\Rd$, with the additional constraint that algorithms must certify the accuracy of their recommendations. We characterize the optimal number of evaluations of any Lipschitz function $f$ to find and certify an approximate maximizer of $f$ at accuracy $\e$. Under a weak assumption on $\cX$, this optimal sample complexity is shown to be nearly proportional to the integral $\int_{\cX} \mathrm{d}\bx/( \max(f) - f(\bx) + \epsilon )^d$. This result, which was only (and partially) known in dimension $d=1$, solves an open problem dating back to \citeyear{hansen1991number}. In terms of techniques, our upper bound relies on a packing bound by \cite{bouttier2020regret} for the Piyavskii-Shubert algorithm that we link to the above integral. We also show that a certified version of the computationally tractable DOO algorithm matches these packing and integral bounds. Our instance-dependent lower bound differs from traditional worst-case lower bounds in the Lipschitz setting and relies on a local worst-case analysis that could likely prove useful for other learning tasks.
\end{abstract}


\section{Introduction}
\label{s:intro}

The problem of optimizing a black-box function $f$ with as few evaluations of $f$ as possible arises in many scientific and industrial fields such as computer experiments \citep{jones1998efficient,richet2013using} or automatic selection of hyperparameters in machine learning \citep{bergstra2011algorithms}. For safety-critical applications, e.g., in aircraft or nuclear engineering, using sample-efficient methods is not enough. \emph{Certifying} the accuracy of the output of the optimization method can be a crucial additional requirement \citep{vanaret2013preventing}.
As a concrete example, \cite{azzimonti2021adaptive} describe a black-box function in nuclear engineering whose output is a $k$-effective multiplication factor, for which a higher value corresponds to a higher nuclear hazard. Certifying the optimization error is a way to certify the worst-case $k$-effective factor, which may be required by safety authorities.

In this paper, we formally study the problem of finding and certifying an $\epsilon$-approximate maximizer of a Lipschitz function $f$ of $d$ variables and characterize the optimal number of evaluations of any such function $f$ to achieve this goal. We start by formally defining the setting.

\subsection{Setting: Zeroth-order Lipschitz Optimization with Error Certificates}
\label{s:algos}

Let $f\colon \cX \to \R$ be a function on a compact non-empty subset $\cX$ of $\Rd$ and $\bxs\in \cX$ a maximizer.

\paragraph{Lipschitz assumption.} We assume that $f$ is Lipschitz with respect to a norm $\lno\cdot$, that is, there exists $L \geq 0$ such that $\babs{f(\bx)-f(\by)}\le L \lno{\bx-\by}$ for all $\bx,\by\in\cX$.
Furthermore, we assume such a Lipschitz bound $L$ to be known.
Even though the smallest \lip{} constant $\Lip(f) := \min\{L'\ge 0 : f \text{ is }L'\text{-\lip{}}\}$ is well defined mathematically, it is rarely known exactly in practical black-box problems.
As a theoretical curiosity, we will briefly discuss the case $L=\Lip(f)$ (i.e., when the best \lip{} constant of the unknown black-box function $f$ is known exactly) in \cref{section:lower:bounds}, but for most of our results, we will make the following more realistic assumption.
\begin{assumption}
\label{a:lip}
For some known \lip{} constant $L$, the function $f\colon \cX \to \R$ belongs to
\begin{equation}
    \label{e:lip}
    \cF_L
:=
    \bcb{
        g \colon \cX \to \R \mid \text{$g$ is \lip{} and } \Lip(g) < L
    } \;.
\end{equation}
\end{assumption}

\vspace{-0.2cm}
The \lip{}ness of $f$ implies the weaker property that $f(\bxs) - f(\bx) \le L\lno{\bxs-\bx}$ for all $\bx\in\cX$, sometimes referred to as \emph{\lip{}ness around a maximizer $\bxs\in \cX$}.
Although this is not the focus of our work, we will mention when our results hold under this weaker assumption.

\paragraph{Online learning protocol.}
We study the case in which $f$ is black-box, i.e., except for the \emph{a priori} knowledge of $L$, we can only access~$f$ by sequentially querying its values at a sequence $\bx_1,\bx_2,\ldots \in \cX$ of points of our choice.
At every round $n \ge 1$, the query point $\bx_n$ can be chosen as a deterministic function of the values $f(\bx_1),\ldots,f(\bx_{n-1})$ observed so far. 
At the end of round $n$, using all the values $f(\bx_1),\ldots,f(\bx_{n})$, the learner outputs two quantities:\vspace{-0.2cm}
\begin{itemize}[leftmargin=2em,itemsep=0.1em]
    \item a recommendation $\bxs_n \in \cX$, with the goal of minimizing the \emph{optimization error} (a.k.a. \emph{simple regret}): $\max(f) - f(\bxs_n)$; 
    \item an \emph{error certificate} $\xi_n \ge 0$, with the constraint to correctly upper bound the optimization error for any $L$-\lip{} function $f\colon \cX \to \R$, i.e., so that $\max(f)-f(\bxs_n)\le \xi_n$.
\end{itemize}

We call \emph{certified algorithm} any algorithm for choosing such a sequence $(\bx_n,\bxs_n,\xi_n)_{n \ge 1}$.

Our goal is to quantify the smallest number of evaluations of $f$ that certified algorithms need in order to find and certify an approximate maximizer of $f$ at accuracy~$\e$.
This objective motivates the following definition.
For any accuracy $\e > 0$, we define the \emph{sample complexity} (that could also be called query complexity) of a certified algorithm $A$ for an $L$-\lip{} function $f$ as
\begin{equation}
    \sigma(A,f,\e)
:=
    \inf \bcb{ n \ge 1 : \xi_n \le \e } \in \{1,2,\ldots\} \cup \{+\iop\} \,.
\label{eq:defsigma}    
\end{equation}

\vspace{-0.2cm}
This corresponds to the first time when we can stop the algorithm while being sure to have an $\e$-optimal recommendation $\bxs_n$.

\subsection{Main Contributions and Outline of the Paper}
\label{sec:outline}

The main result of this paper is a tight characterization (up to a $\log$ factor) of the optimal sample complexity of certified algorithms in any dimension $d \ge 1$, solving a three-decade old open problem raised by \citet{hansen1991number}. More precisely, we prove the following instance-dependent upper and lower bounds, which we later state formally in Theorem~\ref{t:main-thm} of Section~\ref{section:lower:bounds} (see also discussions therein, as well as Propositions~\ref{prop:lower:bound:dim:one} and~\ref{prop:counter:example:Nepsilon} for the limit case $L=\Lip(f)$).

\begin{theorem*}[Informal statement]
Under a mild geometric assumption on $\cX$, there exists a computationally tractable algorithm~A (e.g., $\cdoo$, \cref{alg:DOO}) such that, for some constants $C_d,c_d > 0$ (depending exponentially on the dimension $d$),
any Lipschitz function $f\in\cF_L$ (see \eqref{e:lip}) and any accuracy $\e $,
\begin{equation}
\sigma(A, f, \e) \le
    C_d
    \int_\cX \frac{\mathrm{d}\bx}{\brb{ f(\bxs) - f(\bx) + \e }^d}
 \;,
\label{eq:mainresult-ub}    
\end{equation}
while any certified algorithm $A'$ must satisfy, for all $f\in\cF_L$, and $c\approx c_d (1-\Lip(f)/L)^d /\log (1/\e)$,
\begin{equation}
c
    \int_\cX \frac{\mathrm{d}\bx}{\brb{ f(\bxs) - f(\bx) + \e }^d}
\le
    \sigma(A',f,\epsilon) \;.
\label{eq:mainresult-lb}    
\end{equation}
\end{theorem*}

\vspace{-0.2cm}
In particular, this result extends to any dimension $d \ge 1$ the upper bound proportional to $\int_0^1  \dif x / ( f(x^\star)-f(x)+\e )$ that \citet{hansen1991number} derived in dimension $d=1$ using arguments specific to the geometry of the real line.

\paragraph{Detailed contributions and outline of the paper.} 
We make the following contributions.
\begin{itemize}[leftmargin=2em]
    \item As a warmup, we show in Section~\ref{s:improvedDOO} how to add error certificates to the DOO algorithm (well-known in the more classical zeroth-order Lipschitz optimization setting \emph{without} error certificates, see \citealt{Per-90-OptimizationComplexity,munos2011optimistic}). We then upper bound its sample complexity by the quantity $S_\mathrm{C}(f,\e)$ defined in~\eqref{eq:certified:bound} below. This bound matches a recent bound derived by \citet{bouttier2020regret} for a computationally much more expensive algorithm. In passing, we also slightly improve the packing arguments that \citet{munos2011optimistic} used in the non-certified setting.
    \item In Section~\ref{s:hansen} we show that, under a mild geometric assumption on $\cX$, the complexity measure $S_\mathrm{C}(f,\e)$ is actually proportional to the integral $\int_\cX \mathrm{d}\bx/\brb{ f(\bxs) - f(\bx) + \e }^d$, which implies \eqref{eq:mainresult-ub} above. 
    This extends the bound of \citet{hansen1991number} ($d=1$) to any dimension $d$.
    \item Finally, in Section~\ref{section:lower:bounds}, we prove the instance-dependent lower bound~\eqref{eq:mainresult-lb}, which differs from traditional worst-case lower bounds in the Lipschitz setting. Our proof relies on a \emph{local} worst-case analysis that could likely prove useful for other learning tasks.
\end{itemize}

Some of the proofs are deferred to the \suppl{}, where we also recall useful results on packing and covering numbers (Section~\ref{s:packing-covering}), as well as provide a slightly improved sample complexity bound on the DOO algorithm in the more classical non-certified setting (Section~\ref{s:comparison-non-certif}).

\subsection{Related Works}
\label{sec:existing}

We detail below some connections with the global optimization and the bandit optimization literatures.

\paragraph{Zeroth-order Lipschitz optimization with error certificates.} The problem of optimizing a function with error certificates has been studied in different settings over the past decades. For instance, in convex optimization, an example of error certificate is given by the \emph{duality gap} between primal and dual feasible points (see, e.g., \citealt{BoVa-04-ConvexOptimization}).

In our setting, namely, global zeroth-order Lipschitz optimization with error certificates, most of the attention seems to have been on the very natural (yet computationally expensive) algorithm introduced by \citet{Piy-72-AbsoluteExtremum} and \citet{shubert1972sequential}.\footnote{For the interested reader who is unfamiliar with this classic algorithm, we added some details in Section~\ref{s:missing-lower-2} of the \suppl{}.} In dimension $d=1$, \citet{hansen1991number} proved that its sample complexity $\sigma(\mathrm{PS},f,\e)$ for $L$-Lipschitz functions $f\colon [0,1] \to \R$ is at most proportional to the integral $\int_0^1 \brb{ f(x^\star)-f(x)+\e }^{-1} \dif x$, and left the question of extending the results to arbitrary dimensions open, stating that the task of ``\emph{Extending the results of this paper to the multivariate case appears to be difficult}''. Recently,
writing $\cX_\e := \{\bx \in \cX: \max(f) - f(\bx) \le \e\}$ for the set of $\epsilon$-optimal points, $\cX_{(a,b]} := \{\bx \in \cX: a < \max(f) - f(\bx) \le b\}$ for the set of points in between $a$ and $b$ optimal, and $\cN(E,r)$ for the packing number of a set $E$ at scale $r$ (see \cref{s:notation}),
\citet[Theorem~2]{bouttier2020regret} proved a bound valid in any dimension $d \geq 1$ roughly of this form:
\begin{equation}
S_\mathrm{C}(f,\e) :=
\cN \lrb{ \cX_{\e}, \frac{\e}{L} } + \sum_{k=1}^{\me} \cN \lrb{ \cX_{(\e_k, \e_{k-1}]}, \, \frac{\e_k}{L} } \;,    \label{eq:certified:bound}
\end{equation}
where  the number of terms in the sum is $\me := \bce{ \log_2(\eo/\e) }$ (with $\eo := L \sup_{\bx,\by \in  \cX } \lno{ \bx - \by }$) and the associated scales are given by  $\e_{\me} := \e$ and $\e_k := \eo 2^{-k}$ for all $k \in \{0,1,\ldots,\me-1\}$.

The equivalence we prove in Section~\ref{s:hansen} between $S_\mathrm{C}(f,\e)$ and $\int_\cX \mathrm{d}\bx/(f(\bxs) - f(\bx) + \e)^d$ solves in particular the question left open by \citet{hansen1991number}. The upper bound we prove for the certified DOO algorithm in Section~\ref{s:improvedDOO} also indicates that the bound $S_\mathrm{C}(f,\e)$ and the equivalent integral bound can be achieved with a computationally much more tractable algorithm.\footnote{Tractability refers to running time (number of elementary operations) and not number of evaluations of $f$.} Indeed, the Piyavskii-Shubert algorithm requires at every step $n$ to solve an inner global Lipschitz optimization problem close to the computation of a Voronoi diagram (see discussion in \citealt[Section 1.1]{bouttier2020regret}), hence a running time believed to grow as $n^{\Omega(d)}$ after $n$ function evaluations. On the contrary, as detailed in Remark~\ref{rem:runningtimeDOO} (Section~\ref{s:improvedDOO}), the running time of our certified version of the DOO algorithm is only of the order of $n \log (n)$ after $n$ of evaluations of $f$.

\paragraph{Connections with the bandit optimization literature: upper bounds.}
Our work is also strongly connected to the bandit optimization literature, in which multiple authors studied the global Lipschitz optimization problem with zeroth-order (or \emph{bandit}) feedback, either with perfect (deterministic) or noisy (stochastic) observations. In the deterministic setting considered here, these papers show that though the number $(L/\epsilon)^d$ of evaluations associated to a naive grid search is optimal for worst-case Lipschitz functions (e.g., Thm~1.1.2 by \citealt{nesterov2003introductory}), sequential algorithms can approximately optimize more benign functions with a much smaller number of evaluations. Examples of algorithms with such guarantees in the deterministic setting are the branch-and-bound algorithm by \citet{Per-90-OptimizationComplexity}, the DOO algorithm by \citet{munos2011optimistic} or the LIPO algorithm by \citet{MaVa-17-LipBandits}. Examples of algorithms in the stochastic setting are the HOO algorithm by \citet{bubeck2011x} or the (generic yet computationally challenging) Zooming algorithm by \citet{kleinberg2008multi,kleinberg2019bandits}. More examples and references can be found in the textbooks by \citet{Mun-14-MonteCarloTreeSearch} and \citet{slivkins2019introduction}.

Note however that, except for the work of \citet{bouttier2020regret} mentioned earlier, these bandit optimization papers did not address the problem of \emph{certifying} the accuracy of the recommendations~$\bxs_n$. Indeed, all bounds are related to a more classical notion of sample complexity, namely, the minimum number of queries made by an algorithm $A$ before outputting an $\e$-optimal recommendation:
\begin{equation}
    \zeta(A,f,\e)
:=
    \inf \{ n \geq 1: \max(f) - f( \bxs_n) \leq \e \}  \in \{1,2,\ldots\} \cup \{+\iop\} \;.
\label{eq:defzeta}    
\end{equation}
Though $\zeta(A,f,\e)$ is always upper bounded by $\sigma(A,f,\e)$ defined in~\eqref{eq:defsigma}, these two quantities can differ significantly, as shown by the simple example of constant functions $f$, for which $\zeta(A,f,\e)=1$ but $\sigma(A,f,\e) \approx (L/\e)^d$ since the only way to \emph{certify} that the output is $\e$-optimal is essentially to perform a grid-search with step-size roughly $\e/L$, so as to be sure there is no hidden bump of height more than $\e$. At a high level, the more ``constant'' a function is, the easier it is to recommend an $\e$-optimal point, but the harder it is to certify that such recommendation is actually a good recommendation. See the \suppl{} (Section~\ref{s:comparison-non-certif}) for a comparison of bounds.

Despite this important difference, the bandit viewpoint (using packing numbers instead of more specific one-dimensional arguments) is key to obtain our multi-dimensional integral characterization.

\paragraph{Comparison with existing lower bounds.} Several lower bounds were derived in the bandit Lipschitz optimization setting without error certificates.
When rewritten in terms of the accuracy $\epsilon$ and translated into our deterministic setting, the lower bounds of \citet{Hor-06-UnknownLipschitz} (when $d^\star = d/2$) and of \citet{bubeck2011x} (for any $d^\star$) are of the form $\inf_A \sup_{f \in \mathcal{G}_{d^\star}} \zeta(A,f,\e) \gtrsim (1/\e)^{d^\star}$, where $\mathcal{G}_{d^\star}$ is the subset of $L$-Lipschitz functions with near-optimality dimension at most $d^\star$. These are worst-case (minimax) lower bounds.

On the contrary, our instance-dependent lower bound \eqref{eq:mainresult-lb} quantifies the minimum number of evaluations to certify an $\epsilon$-optimal point \emph{for each function $f \in \cF_L$}. 
Note that here the certified setting enables to obtain meaningful instance-dependent lower bounds, 
whereas their non-certified counterparts would be trivial (equal to one).
Our proof relies on a \emph{local} worst-case analysis in the same spirit as for distribution-dependent lower bounds in stochastic multi-armed bandits (see, e.g., Theorem~16.2 in \citealt{LS20-banditalgos}), yet for continuous instead of finite action sets. We believe this lower bound technique should prove useful for other learning tasks.

\subsection{Recurring Notation}
\label{s:notation}

This short section contains a summary of all the notation that we use in the paper and can be used by the reader for easy referencing.
We denote the set of positive integers $\{1,2,\ldots\}$ by $\Ns$ and let $\N := \N^* \cup\{0\}$.
For all $n\in \Ns$, we denote by $[n]$ the set of the first $n$ integers $\{1,\ldots,n\}$.
We denote the \leb{} measure of a (\leb{}-measurable) set $E$ by $\vol(E)$ and refer to it simply as its \emph{volume}.
For all $\rho>0$ and $\bx \in \R^d$, we denote by $B_\rho(\bx)$ the closed ball with radius $\rho$ centered at $\bx$, with respect to the arbitrary norm $\lno \cdot$ that is fixed throughout the paper. We also write $B_\rho$ for the ball with radius $\rho$ centered at the origin and denote by $v_\rho$ its volume.

$\Lip(f)$ denotes the smallest \lip{} constant of our target $L$-\lip{} function $f\colon \cX \to \R$.
The set of its $\epsilon$-optimal points is denoted by $\cX_\e := \{\bx \in \cX: \max(f) - f(\bx) \le \e\}$, its complement (i.e., the set of $\e$-\emph{suboptimal points}) by $\cX_\e^c$, and for all $0\le a<b$, the $(a,b]$-\emph{layer} (i.e., the set of points that are $b$-optimal but $a$-suboptimal) by
$
    \cX_{(a,b]}
:=
	\cX_a^c \cap \cX_b
=
	\bcb{ \bx \in  \cX   :  a < f(\bxs)-f(\bx)\le b}
$.
Since $f$ is $L$-\lip{}, every point in $\cX$ is $\eo $-optimal, with $\eo$ defined by
$
	\eo 
:=
	L \max_{\bx,\by \in  \cX } \lno{ \bx - \by }
$.
In other words, $\cX_{\eo }= \cX $. For this reason, without loss of generality, we will only consider values of accuracy $\e$ smaller than or equal to $\eo$.

For any bounded set $E\s\Rd$ and all $r>0$,
the $r$-\emph{packing number} of $E$ is the largest cardinality of an $r$-packing of $E$, that is,
$
	\cN(E,r)
:= 
	\sup \bcb{
		k \in \Ns  :  \exists \bx_1, \ldots, \bx_k \in E, \min_{i\neq j} \lno{ \bx_i - \bx_j } > r
	}
$
if $E$ is nonempty, zero otherwise;
the $r$-\emph{covering number} of $E$ is the smallest cardinality of an $r$-covering of $E$, i.e.,
$
	\cM(E,r)
:=
	\min \bcb{
		k \in \Ns  :  \exists \bx_1,\ldots,\bx_k \in \Rd, \forall \bx \in E, \exists i\in [ k ], \lno{ \bx-\bx_i}\le r
	}
$
if $E$ is nonempty, zero otherwise.
Well-known and useful properties of packing (and covering) numbers are recalled in Section~\ref{s:packing-covering} of the \suppl{}.

\section{Warmup: Certified DOO Has Sample Complexity \texorpdfstring{$\SC(f,\e)$}{SC}}
\label{s:improvedDOO} \label{sec:DOOdef} \label{sec:DOOnoncertified} \label{sec:DOOcertified}

In this section, we start by adapting the well-known DOO algorithm \citep{Per-90-OptimizationComplexity,munos2011optimistic} to the certified setting. We then bound its sample complexity by the quantity $\SC(f,\e)$ defined in \cref{eq:certified:bound}.
In passing, we slightly improve the packing arguments used by \cite{munos2011optimistic} in the non-certified setting (\suppl{}, Section~\ref{s:comparison-non-certif}).
In \Cref{section:lower:bounds}, we will prove that this bound is optimal (up to logarithmic factors) for certified algorithms.

The certified DOO algorithm ($\cdoo$, \cref{alg:DOO}) is defined for a fixed $K \in \mathbb{N}^\star$, by an infinite sequence of subsets of $\cX$ of the form $(X_{h,i})_{h \in \mathbb{N},i=0,\ldots,K^h-1}$, called \emph{cells}. 
For each $h \in \mathbb{N}$, the cells $X_{h,0},\ldots,X_{h,K^h-1}$ are non-empty, pairwise disjoint, and their union contains $\mathcal{X}$.
The sequence $(X_{h,i})_{h \in \mathbb{N},i=0,\ldots,K^h-1}$ is associated with a $K$-ary tree in the following way. 
For any $h \in \mathbb{N}$ and $j \in \{0,\ldots,K^h-1\}$, there exist $K$ distinct $i_1,\ldots,i_K \in \{ 0,\ldots,K^{h+1}-1 \}$ such that $X_{h+1,i_1}, \ldots , X_{h+1,i_K}$ form a partition of $X_{h,j}$. We call $(h+1,i_1),\ldots,(h+1,i_K)$ the \emph{children} of $(h,j)$.
To each cell $X_{h,i}$ ($h \in \mathbb{N}$, $i\in \{0,\ldots,K^h-1\}$) is associated a \emph{\representative{}} $\bx_{h,i} \in X_{h,i}$, which can be thought of as the ``center'' of the cell. 
We assume that feasible cells have feasible \representative{}s, i.e., that $X_{h,i} \cap \cX \neq \varnothing$ implies $\bx_{h,i} \in \cX$. 
The two following assumptions prescribe a sufficiently good behavior of the sequences of cells and \representative{}s. 

\begin{assumption} \label{assumption:DOO:small:cells}
There exist two positive constants $\delta \in (0, 1)$ and $R > 0$ such that, for any cell $X_{h,i}$ ($h \in \mathbb{N}$, $i=0,\ldots,K^h-1$) and all $\boldsymbol u, \boldsymbol v \in X_{h,i}$, it holds that
$
    \lno{ \boldsymbol u - \boldsymbol v } \leq R \delta^h
$.
\end{assumption}

\begin{assumption} \label{assumption:DOO:well-separated}
There exists $\nu > 0$ such that, with $\delta$ as in Assumption \ref{assumption:DOO:small:cells}, for any $h \in \mathbb{N}$, $i=0,\ldots,K^h-1$, $h' \in \mathbb{N}$, $i'=0,\ldots,K^{h'}-1$, with $(h,i) \neq (h',i')$,
$
\lno{ \bx_{h,i} - \bx_{h',i'} } \geq \nu \delta^{\max(h,h')}
$.
\end{assumption}

The classic Assumption~\ref{assumption:DOO:small:cells} is simply stating that diameters of cells decrease geometrically with the depth of the tree. 
Assumption~\ref{assumption:DOO:well-separated}, which is key for our improved analysis, is slightly stronger than the corresponding one in \citet{munos2011optimistic}, yet very easy to satisfy. 
Indeed, one can prove that for any compact $\mathcal{X}$, it is always possible to find a sequence of cells and \representative{}s satisfying \cref{assumption:DOO:small:cells,assumption:DOO:well-separated}. For instance, if $\mathcal{X}$ is the unit hypercube $[0,1]^d$ and $\lno \cdot$ is the supremum norm $\lno \cdot _{\infty}$, we can define cells by bisection, letting $K=2^d$, $X_{h,i}$ be a hypercube of edge-length $2^{-h}$, and $\bx_{h,i}$ be its center (for $h \in \mathbb{N}$ and $i=0,\ldots,2^{dh}-1$). In this case, \cref{assumption:DOO:small:cells,assumption:DOO:well-separated} are satisfied with $R=1$ and $\delta = \nu = 1/2$.

\begin{algorithm2e}
\DontPrintSemicolon
\SetAlgoNoLine
\SetAlgoNoEnd
\LinesNumbered
\SetKwInput{KwIn}{input}
\SetKwInput{kwInit}{initialization}
\SetKwInput{kwRep}{repeat}
\KwIn{$\mathcal{X}$, $L$, $K$, $\delta$, $R$, cells $(X_{h,i})_{h \in \mathbb{N},i=0,\ldots,K^h-1}$, \representative{}s $(\bx_{h,i})_{h \in \mathbb{N},i=0,\ldots,K^h-1}$}
\kwInit{ 
    let $n \gets 1$ and $\mathcal{L}_1 \gets \{ (0,0) \}$
}
pick the first query point $\bx_1 \gets \bx_{0,0}$\;
observe the value $f(\bx_1)$\;
output recommendation $\bxs_1 \gets \bx_1$ and error certificate $\xi_1 \gets LR$\label{state:firstCertif}\;
\For{iteration $k=1,2,\ldots$}{%
let 
$
    (h^\star,i^\star)
\in 
    \argmax_{(h,i) \in \mathcal{L}_n}
    \bcb{
        f(\bx_{h,i}) + L R \delta^h
    }
$ \label{eq:DOO:select:cell} \tcp*{ties broken arbitrarily}
let $\mathcal{L}_+$ be the set of the $K$ children of $(h^\star,i^\star)$\label{state:split}\;
\For{\textbf{\emph{each}} child $(h^\star+1,j) \in \mathcal{L}_+$ of $(h^\star,i^\star)$}{
    \If{%
        $X_{h^\star+1,j} \cap \cX \neq \varnothing$%
    }
    {%
        let $n \gets n+1$ and $\mathcal{L}_n \gets \mathcal{L}_{n-1} \cup \{ (h^\star+1,j) \}$\label{state:dooaddchildren}\;
        pick the next query point $\bx_n \gets \bx_{h^\star+1,j}$\label{state:doopick}\;
        observe the value $f(\bx_n)$\label{state:dooobserve}\;
        output a recommendation $\bx_n^\star \in \mathrm{argmax}_{\bx \in \{\bx_1,\ldots,\bx_n \}} f(\bx)$\label{state:doorecommend}\;
        output the error certificate 
        $
            \xi_n 
        \gets
            f(\bx_{h^\star,i^\star}) + 
            L R \delta^{h^\star}
            - 
            f(\bxs_n)
        $\label{state:doocertificate}\;
    }
}
remove $(h^\star,i^\star)$ from $\mathcal{L}_n$\label{state:dooremoveparent}
}
\caption{\label{alg:DOO} Certified DOO ($\cdoo$)}
\end{algorithm2e}

Our certified version of the DOO algorithm ($\cdoo$, \cref{alg:DOO}) maintains a set of indices of \emph{active} cells $\mathcal{L}_n$ throughout rounds $n$. 
During each iteration $k$, it selects the index of the most promising active cell $(h^\star,i^\star)$ (\cref{eq:DOO:select:cell}) and splits it into its $K$ children $\cL_+$ (\cref{state:split}). 
Then, it sequentially picks the \representative{}s of the cells corresponding to each of these children (\cref{state:doopick}), observes the value of the target function $f$ at these points (\cref{state:dooobserve}), recommends the point with the highest observed value of $f$ (\cref{state:doorecommend}), and outputs a certificate 
$
    \xi_n 
= 
    \brb{ f(\bx_{h^\star,i^\star}) + 
    L R \delta^{h^\star} }
    - 
    f(\bxs_n)
$
(\cref{state:doocertificate})
that is the difference between an upper bound on $\max (f)$ and the currently recommended value $f(\bxs_n)$.
In the meantime, all children in $\cL_+$ are added to the set of active indices $\cL_n$ (\cref{state:dooaddchildren}), and the current iteration is concluded by removing $(h^\star,i^\star)$ from $\cL_n$ (\cref{state:dooremoveparent}), now that it has been replaced by its refinement $\cL_+$.

\begin{remark}
\label{rem:runningtimeDOO}
The running-time of $\cdoo$ (ignoring the cost of calling the function $f$) is driven by the computation of the recommendation $\bx_n^\star \in \mathrm{argmax}_{\bx \in \{\bx_1,\ldots,\bx_n \}} f(\bx)$ (\cref{state:doorecommend}) and the search of the index of the most promising active cell
$
    (h^\star,i^\star)
\in 
    \argmax_{(h,i) \in \mathcal{L}_n}
    \{
        f(\bx_{h,i}) + L R \delta^h
    \}
$ (\cref{eq:DOO:select:cell}).
The recommendation $\bx_n^\star$ can be computed sequentially in constant time (by comparing the new value $f(\bx_n)$ with the current maximum).
In \cref{eq:DOO:select:cell}, the leaf $(h^\star,i^\star)$ to be split at iteration $k$ can be computed sequentially in logarithmic time (using a max-heap structure).
Therefore, the running time of $\cdoo$ is of order $n \log (n)$ in the number $n$ of evaluations of the function $f$.
\end{remark}

The next proposition shows that the sample complexity \textcolor{black}{(see \eqref{eq:defsigma})} of the certified DOO algorithm is upper bounded (up to constants) by the instance-dependent quantity $S_\mathrm{C}(f,\e)$ introduced in \cref{eq:certified:bound}.

\begin{proposition}
\label{prop:DOO:certified}
If \cref{assumption:DOO:small:cells,assumption:DOO:well-separated} hold, then \cref{alg:DOO} is a certified algorithm and letting 
$\textcolor{black}{a_d} := \textcolor{black}{2+}
K
\brb{
\mathbf{1}_{ \nu/ R \geq 1 }
+
\mathbf{1}_{ \nu/ R < 1 }
(
\fracc{4 R}{\nu}
)^d
}$, 
its sample complexity satisfies, for all \lip{} functions $f \in \cF_L$ \textcolor{black}{(see \eqref{e:lip})} and any accuracy $\e\in (0,\eo ]$,
\[
    \sigma(\cdoo, f, \e) \le \textcolor{black}{a_d} \SC(f,\e)\;.
\]
\end{proposition}

The proof is postponed to Section~\ref{s:missing-upper-DOO} of the \suppl{} and shares some arguments with those of \citet{Per-90-OptimizationComplexity} and \citet{munos2011optimistic}, originally written for the non-certified setting.
The key change is to partition the values of $f$, instead of its domain $\cX$ at any depth $h$ of the tree (see \citealt{munos2011optimistic}), when counting the representatives selected at all levels. The idea of using layers $\cX_{\left(\e_i, \ \e_{i-1} \right]}$ was already present in \citet{kleinberg2008multi,kleinberg2019bandits} and \citet{bouttier2020regret} for more computationally challenging algorithms (see discussion in \Cref{sec:existing}).

\section{Characterization of \texorpdfstring{$\SC(f,\e)$}{SC(f,epsilon)}}
\label{s:hansen}

Earlier, we mentioned that the quantity $\SC(f,\e)$ introduced in \cref{eq:certified:bound} upper bounds the sample complexity of several certified algorithms, such as $\cdoo$ or \piya{}-Shubert.
In this section, we provide a characterization of this quantity in terms of a much cleaner and integral expression.

This result is inspired by \citet{hansen1991number}, that in dimension $d=1$, derive an elegant bound on the sample complexity $\sigma(\mathrm{PS},f,\e)$ of the certified \piya{}-Shubert algorithm for any \lip{} function $f$ and accuracy $\e$.
They proved that $\sigma(\mathrm{PS},f,\e)$ is upper bounded by $\int_0^1  \dif x / ( f(x^\star)-f(x)+\e )$ up to constants.
However, the authors rely heavily on the one-dimensional assumption and the specific form of the \piya{}-Shubert algorithm in this setting, stating that the task of ``\emph{Extending the results of this paper to the multivariate case appears to be difficult}''.
In this section, we show an equivalence between $\SC(f,\e)$ and
this type of integral bound
in any dimension $d$.
Putting this together with a recent result of \cite{bouttier2020regret} (which proves that up to constants, $\sigma(\mathrm{PS},f,\e) \le \SC(f,\e) $) solves the long-standing problem raised by \cite{hansen1991number} three decades ago.

To tame the wild spectrum of shapes that compact subsets may have, we will assume that $\cX$ satisfies the following mild geometric assumption.
At a high level, it says that a constant fraction of each (sufficiently small) ball centered at a point in $\cX$ is included in $\cX$.
This removes sets containing isolated points or cuspidal corners and, \textcolor{black}{as can be shown quickly,} includes most domains that are typically used, such as \textcolor{black}{finite unions of convex sets with non-empty interiors}. 
This natural assumption is weaker than the classic rolling ball assumption from the statistics literature \citep{cuevas2012statistical,walther1997granulometric} and has already been proved useful in the past \citep{HKM20-SmoothContextualBandits}. 

\begin{assumption}
\label{ass:geomCx}
   There exist two constants $r_0>0,\gamma\in(0,1]$ such that, for any $\bx\in\cX$ and all $r\in(0,r_0)$, $\vol\brb{ B_r(\bx) \cap \cX } \ge \gamma v_r$.
\end{assumption}

\textcolor{black}{
Note that, above, $\gamma$ may implicitly depend on $d$. For instance if $\cX$ is an hypercube, then $\gamma$ is at most $2^{-d}$.} 
We can now state the main result of this section.
Its proof relies on some additional technical results that are deferred to the \suppl{}. \textcolor{black}{Recall that $ \SC(f,\e)$ was defined in \eqref{eq:certified:bound}.}

\begin{theorem}
\label{t:hansen}
For any \lip{} function $f\in \cF_L$ \textcolor{black}{(see \eqref{e:lip})}, if Assumption~\ref{ass:geomCx} holds with $r_0 > \fracc{\eo}{2L},\gamma\in(0,1]$,\footnote{We actually prove a stronger result. The first inequality holds more generally for any $f$ that is $L$-\lip{} around a maximizer and \leb{}-measurable, and does not require $\cX$ to satisfy \cref{ass:geomCx}.} then there exist $\textcolor{black}{c_d},\textcolor{black}{C_d}>0$ (e.g., $\textcolor{black}{c_d} := 1/v_{\nicefrac{1}{L}}$ and $\textcolor{black}{C_d} := 1/(\gamma v_{\fracc{1}{128L}})$) such that, for all $\e\in (0,\eo]$,
\[
    \textcolor{black}{c_d}
    \int_\cX \frac{\mathrm{d}\bx}{\brb{ f(\bxs) - f(\bx) + \e }^d}
\le
    \SC(f,\e)
    \le 
    \textcolor{black}{C_d}
    \int_\cX \frac{\mathrm{d}\bx}{\brb{ f(\bxs) - f(\bx) + \e }^d}
    \;.
\]
\end{theorem}

\textcolor{black}{Remark that in Theorem \ref{t:hansen}, the integral is multiplied by $L^d$ in the two inequalities. This is to compensate for the fact that this integral does not depend on $L$, while $\SC(f,\e)$ scales like $L^d$ for a fixed $f$ as $L \to \infty$.}

\begin{proof}
Fix any $\e \in (0,\eo]$ and
recall that $\me := \bce{ \log_2(\nicefrac{\eo}{\e}) }$, $\e_{\me} := \e$, and for all $k \le \me -1$, $\e_k := \eo 2^{-k}$.
Partition the domain of integration $\cX$ into the following $\me + 1$ sets: 
the set of $\e$-optimal points $\cX_{\e}$ and the $\me$ layers $\cX_{(\e_k, \ \e_{k-1}]}$, for $k \in [\me]$.
We begin by proving the first inequality:
\begin{align*}
    \int_\cX \frac{\mathrm{d}\bx}{\brb{ f(\bxs) - f(\bx) + \e }^d}
&
\le
    \frac{\vol (\cX_\e) }{\e^d}
    + \sum_{k=1}^{\me}\frac{ \vol \brb{ \cX_{(\e_k, \ \e_{k-1}]} } }{ (\e_k + \e)^d}
\\
&
\le
    \frac{\cM \brb{ \cX_{\e}, \ \frac{\e}{L} } \cdot v_1 \brb{ \frac{\e}{L} }^d }{ \e^d }
    + \sum_{k=1}^{\me}\frac{ \cM \brb{ \cX_{(\e_k, \ \e_{k-1}]}, \ \frac{\e_k }{L} } \cdot v_1 \brb{ \frac{\e_k }{L} }^d }{ \e_k ^d }
\\
&
\le
    \frac{v_1}{L^d} \lrb
    {
        \cN \lrb{ \cX_{\e}, \ \frac{\e}{L} } + 
        \sum_{k=1}^{\me} \cN \lrb{ \cX_{(\e_k, \ \e_{k-1}]}, \ \frac{\e_k }{L} }
    }\;,
\end{align*}
where the first inequality follows by lower bounding $f(\bxs) - f$ with its infimum on the partition, the second one by dropping $\e>0$ from the second denominator and upper bounding the volume of a set with the volume of the balls of a smallest $\nicefrac{\e_k}{L}$-cover, and the last one by the fact that covering numbers are always smaller than packing numbers (we recall this known result in the \suppl{}, \cref{l:wainwright}, \cref{eq:wainwright}).
This proves the first part of the theorem.

For the second one, we have
\begin{align*}
&
    \int_\cX \frac{\mathrm{d}\bx}{\brb{ f(\bxs) - f(\bx) + \e }^d}
\ge
    \frac{\vol (\cX_{\e}) }{(\e + \e)^d}
    + \sum_{k=1}^{\me}\frac{ \vol \brb{ \cX_{(\e_k, \ \e_{k-1}]} } }{ (\e_{k-1} + \e)^d}
\\
&
\hspace{28.61375pt}
\ge
    \frac{1}{2^d}
    \frac{\vol (\cX_{\e}) }{\e^d}
    + 
    \frac{1}{4^d}
    \sum_{k=1}^{\me}\frac{ \vol \brb{ \cX_{(\e_k, \ \e_{k-1}]} } }{ \e_{k}^d}
\ge
    \frac{1}{32^d} \lrb{
        \frac{\vol (\cX_{2 \e }) }{ \e^d }
        + \sum_{k=1}^{\me}\frac{ \vol \lrb{ \cX_{ \left(\frac{1}{2} \e_k, \ 2 \e_{k-1}\right]} } }{ \e_{k-1}^d}
    }
\\
&
\hspace{28.61375pt}
\ge
    \frac
    { 
        \cN\lrb{\cX_{\e}, \ \frac{\e}{L}} \vol\lrb{ \frac{\e}{2L} B_1 } 
    }
    {
        (32 \, \e)^d / \gamma 
    }
    + \sum_{k=1}^{\me} \frac
    { 
        \cN\lrb{\cX_{(\e_k, \ \e_{k-1}]}, \ \frac{\e_k}{L}} \vol\lrb{ \frac{\e_k}{2L} B_1 } 
    }
    { 
        (32 \, \e_{k-1} )^d / \gamma 
    }
\\
&
\hspace{28.61375pt}
\ge
    \gamma v_{\nicefrac{1}{64L}}
    \cN\lrb{\cX_{\e}, \ \frac{\e}{L}}
    + 
    \gamma v_{\nicefrac{1}{128L}} \sum_{k=1}^{\me} 
    \cN\lrb{\cX_{(\e_k, \ \e_{k-1}]}, \ \frac{\e_k}{L}}
    \;,
\end{align*}
where the first inequality follows by upper bounding $f(\bxs) - f$ with its supremum on the partition, the second one by $\e \le \e_{k-1}$ (for all $k \in [\me+1]$) and $\e_{k-1}\le 2\e_k$ (for all $k \in [\me]$), the third one by lower bounding the sum of disjoint layers with that of overlapping ones (proved in the \suppl{}, \Cref{l:overlap}), and the fourth one by the elementary inclusions 
$\cX_{2 \e } \supseteq \cX_{\frac{3}{2} \e}$
and
$\cX_{\left(\frac{1}{2} \e_k, \ 2\e_{k-1}\right]} \supseteq \cX_{\left(\frac{1}{2} \e_k, \ \frac{3}{2} \e_{k-1}\right]}$ (for all $k \in [\me]$) followed by a relationship between packing numbers and volumes (proved in the \suppl{}, \Cref{p:pack-vol}) that holds under \cref{ass:geomCx}.
\end{proof}

\section{Optimality of \texorpdfstring{$\SC(f,\e)$}{SC(f,epsilon)}} \label{section:lower:bounds}

We begin this section by proving an $f$-dependent lower bound on the sample complexity of certified algorithms that matches the upper bound $\SC(f,\e)$ on the sample complexity of the $\cdoo$ algorithm (\Cref{prop:DOO:certified}), up to a $\log(1/\e)$ term\textcolor{black}{, a dimension-dependent constant, and the term $(1-\Lip(f)/L)^d$}. 

The proof of this result differs from those of traditional worst-case lower bounds.
The idea is to build a \emph{local} worst-case analysis, introducing a weaker notion of sample complexity $\tau$ that is smaller than $\sigma(A,f,\e)$. Then we further lower bound this quantity by finding adversarial perturbations of the target function $f$ with the following opposing properties. First, these perturbations have to be similar enough to $f$ so that running $A$ on them would return the same recommendations for sufficiently many rounds. Second, they have to be different enough from $f$ so that enough rounds have to pass before being able to certify sub-$\e$ accuracy. 
We recall that $\me := \bce{ \log_2(\eo/\e) }$ \textcolor{black}{and that $ \SC(f,\e)$ was defined in \eqref{eq:certified:bound}.}

\begin{theorem} \label{thm:lower:bound:sample:dependent}
The sample complexity of any certified algorithm $A$ satisfies
\[
\sigma(A,f,\epsilon) 
>
\frac{\textcolor{black}{ c'_d (1-\Lip(f)/L)^d} }{1+m_{\epsilon}}
S_\mathrm{C}(f,\e)
\]
for some constant $c'_d$ (that can be taken as $c'_d = 2^{-2} 2^{-7d}$), any \lip{} function $f\in \cF_L$ \textcolor{black}{(see \eqref{e:lip})} and all $\e \in (0,\eo]$.
\end{theorem}

\begin{proof}
Fix any $f\in \cF_L$ and an accuracy $\e\in (0,\eo]$.
We begin by defining the tightest error certificate that a certified algorithm $A$ could return based on its first $n$ observations of $f$. Formally,
\[
    \err_n(A)
:=
    \sup\bcb{\max(g) - g(\bxs_n) : g \text{ is } L\text{-\lip{} and } f(\bx) = g(\bx) \text{ for all } x\in \{\bx_1,\ldots,\bx_n\}}
\]
where $\bx_i = \bx_i(A,f)$ and $\bxs_i=\bxs_i(A,f)$ are the query and recommendation points chosen by $A$ at time $i$ when run on $f$ (to lighten the notation, we omit the explicit dependencies on $A$ and $f$ of $\bx_i$ and $\bxs_i$).
In particular, $\max(f)-f(\bxs_n)\le \err_n(A)$.
Based on this quantity, we define a corresponding notion of optimal sample complexity $\tau$ as the smallest number of rounds $n$ that the best certified algorithm $A'$ needs in order to guarantee that $\err_n(A')\le \e$. Formally,
\begin{equation} \label{eq:def:tau}
\textstyle{
    \tau
:=
    \min
    \bcb{
    n \in \Ns 
    :
    \inf_{A'}  \brb{ \err_n(A') } \leq \epsilon
    } \;,
}
\end{equation}
where the infimum is over all certified algorithms $A'$.
It is immediate to prove that $\tau$ is finite, by considering an algorithm that queries a dense sequence of points (independently of the observed function values) and outputs a recommendation that maximizes the observed values.

Crucially, $\tau$ lower bounds the sample complexity $\sigma(A,f,\e)$ of any certified algorithm $A$.
At a high level, this makes intuitive sense because $\tau$ guarantees a weaker property: that the best certificate of the best algorithm is small, while $\sigma(A,f,\e)$ controls the certificate of the specific algorithm $A$.
We defer a formal proof of this statement to the \suppl{}, Section~\ref{s:missing-details-lower}.

Since 
$
    \sigma(A,f,\e) \ge \tau
$,
to prove the theorem
it is sufficient to show that, \textcolor{black}{with $c := 2^{-2} 2^{-7d} (1-\Lip(f)/L)^d$}, 
\[
    \tau 
> 
    \frac{ c }{1+m_{\epsilon}}
    S_\mathrm{C}(f,\e) \;.
\]
Let $\textcolor{black}{Q} := 16 L / \brb{ L - \Lip(f) }$.
If $S_\mathrm{C}(f,\e) / (1+m_{\epsilon}) < 3 (8\textcolor{black}{Q})^d$, then 
$\tau \ge 1 > 3/4 > c S_\mathrm{C}(f,\e) / (1+m_{\epsilon})$.
Assume from now on that $S_\mathrm{C}(f,\e) / (1+m_{\epsilon}) \geq 3 (8\textcolor{black}{Q})^d$.

The idea now is to upper bound the sum of $1+\me$ packing numbers that define $\SC(f,\e)$ in \eqref{eq:certified:bound} with the largest one multiplied by $1+\me$.
This way, $\SC(f,\e) / (1+ m_{\epsilon})$ can be controlled by (an upper bound of) the largest summand in $S_\mathrm{C}(f,\e)$.
Let $\tilde{\epsilon}$ be the scale achieving the maximum in \eqref{eq:certified:bound}, that is
\[
\tilde{\epsilon} =  
\begin{cases}
\e \;,
& \text{if }
 \cN \lrb{ \cX_{ \e}, \ \frac{\e}{L} }
 \geq 
  \max_{i \in \{1,\ldots,\me\}} \cN \lrb{ \cX_{\left(\e_i, \ \e_{i-1} \right]}, \ \frac{\e_{i} }{ L} } \;,
  \\
\e_{i^\star-1} \;,
 & \text{otherwise, where }
 i^\star \in \argmax_{i \in \{1,\ldots,m_{\epsilon}\} }
 \cN \lrb{ \cX_{\left(\e_i, \ \e_{i-1} \right]}, \ \frac{\e_{i} }{L} } \;.
\end{cases}
\]
Since 
$
\cN \lrb{ \cX_{ \e}, \ \e/L }
\leq 
\cN \lrb{ \cX_{ \e}, \ \e/2L }
$
and
$ \cN \lrb{ \cX_{\left(\e_i, \ \e_{i-1} \right]}, 
\ \e_{i} / L } \leq \cN \lrb{ \cX_{\ \e_{i-1}}, \ \e_{i-1} /2 L }$,
we then have
$
S_\mathrm{C}(f,\e)
\leq (m_{\epsilon}+1)
 \cN \lrb{ \cX_{ \tilde{\epsilon} }, \ \tilde{\epsilon}/2L}$. 
Let now $n \leq c S_\mathrm{C}(f,\e) /(1+ m_{\epsilon})$. We then have
$n \leq 
c \cN \lrb{ \cX_{ \tilde{\epsilon} }, \ \tilde{\epsilon}/2L}$.
From a known property of packing numbers (\suppl, \cref{lem:changing radius}),
\[
\textstyle{
\cN \lrb{ \cX_{ \tilde{\epsilon} }, \ \frac{\textcolor{black}{Q} \tilde{\epsilon}}{L} }
\geq \left( \frac{1}{8\textcolor{black}{Q}} \right)^d  \cN \lrb{ \cX_{ \tilde{\epsilon} }, \ \frac{\tilde{\epsilon}}{2L} }
\geq 
\left( \frac{1}{8\textcolor{black}{Q}} \right)^d
 \frac{S_\mathrm{C}(f,\e)}{m_{\epsilon} +1} \geq 3 \;,
}
\]
by our initial assumption.
Then we have
$
n \leq c (8\textcolor{black}{Q})^d \cN \lrb{ \cX_{ \tilde{\epsilon} }, \ \textcolor{black}{Q} \tilde{\epsilon}/L}$.
Since $c (8\textcolor{black}{Q})^d =1/4$, we thus obtain
$n \leq \cN \lrb{ \cX_{ \tilde{\epsilon} }, \ \textcolor{black}{Q} \tilde{\epsilon}/L} - 2$.

Consider a certified algorithm $A$ for $L$-\lip{} functions.
Fix a $(\textcolor{black}{Q} \tilde{\epsilon} / L)$-packing $\tilde \bx_1 , \ldots , \tilde \bx_N$ of $\cX_{ \tilde{\epsilon}}$ with cardinality $N = \cN \lrb{ \cX_{ \tilde{\epsilon} }, \ \textcolor{black}{Q} \tilde{\epsilon}/L }$. Then the open balls of centers $\tilde \bx_1 , \ldots , \tilde \bx_N$ and radius $\textcolor{black}{Q} \tilde{\epsilon} / 2L$ are disjoint and two of them, with centers, say, $\tilde{\bx}_1$ and $\tilde{\bx}_2$, do not contain any of the points $\bx_1, \ldots , \bx_n$ queried by $A$ when it is run on $f$. Let, for $\bx \in \cX$, 
\[
\textstyle{
    g_{\tilde{\epsilon}}(\bx)
:=
    \brb{ 8 \tilde{\epsilon} - \frac{16L}{\textcolor{black}{Q}} \lno{ \bx - \tilde{\bx}_1 } }
    \I \bcb{ \bx \in \cX \cap B_{\textcolor{black}{Q} \tilde{\epsilon} / 2L}(\tilde{\bx}_1 ) } \;.
}
\]
Then $g_{\tilde{\epsilon}}(x)$ is $16L / \textcolor{black}{Q} = L - \Lip(f)$ Lipschitz. Hence $f + g_{\tilde{\epsilon}}$ and $f - g_{\tilde{\epsilon}}$ are $L$-Lipschitz. 
Note that $f$, $f + g_{\tilde{\epsilon}}$ and $f - g_{\tilde{\epsilon}}$ coincide on the points $\bx_1, \ldots , \bx_n$ that $A$ queries when it is run on $f$. 
As a consequence, $A$ queries the same points and returns the same recommendation $\bxs_n$ when it is run on any of the three functions $f$, $f + g_{\tilde{\epsilon}}$, and $f - g_{\tilde{\epsilon}}$. 

Consider first the case $\bxs_n \in B( \tilde{\bx}_1 , \textcolor{black}{Q} \tilde{\epsilon} / 4L )$. Then, we have, by definition of $g_{\tilde{\epsilon}}$ and the fact that $\tilde{\bx}_2 \in \mathcal{X}_{\tilde{\epsilon}}$,
$
f( \tilde{\bx}_2 ) - g_{\tilde{\epsilon}}( \tilde{\bx}_2 )
- f( \bxs_n ) + g_{\tilde{\epsilon}}( \bxs_n )
\geq 
- \tilde{\epsilon} 
+ 8 \tilde{\epsilon}
- 
\frac{16L}{\textcolor{black}{Q}} \frac{\textcolor{black}{Q} \tilde{\epsilon}}{4L}
=
3 \tilde{\epsilon}
$.

Now consider the case $\bxs_n \notin B( \tilde{\bx}_1 , \textcolor{black}{Q} \tilde{\epsilon} / 4L )$. Then, we have, by definition of $g_{\tilde{\epsilon}}$ and the fact that $\tilde{\bx}_1 \in \mathcal{X}_{\tilde{\epsilon}}$,
$
f( \tilde{\bx}_1 ) + g_{\tilde{\epsilon}}( \tilde{\bx}_1 )
- f( \bxs_n ) - g_{\tilde{\epsilon}}( \bxs_n )
\geq 
- \tilde{\epsilon} 
+ 8 \tilde{\epsilon}
- 
8 \tilde{\epsilon}
+
\frac{16L}{\textcolor{black}{Q}} \frac{\textcolor{black}{Q} \tilde{\epsilon}}{4L}
=
3 \tilde{\epsilon}
$.

Therefore, in both cases $\err_n(A) \ge 3 \tilde{\epsilon} > \e $. 
Repeating the same construction from any other certified algorithm $A'$, we obtain that $\inf_{A'} \brb{ \err_n(A') } > \e $. 
Since this has been shown for any $n \leq c S_\mathrm{C}(f,\e) /(1+ m_{\epsilon})$, by definition of $\tau$ we can conclude that $\tau
> 
c S_\mathrm{C}(f,\e) /(1+m_{\epsilon})$.
\end{proof}

Putting together \Cref{prop:DOO:certified,thm:lower:bound:sample:dependent} shows that the sample complexity of $\cdoo$ applied to any \lip{} function $f\in \cF_L$ is of order $\SC(f,\e)$ and no certified algorithm $A$ can do better, \emph{not even if $A$ knows $f$ exactly}.
This is a striking difference with the classical non-certified setting in which the best algorithm for each fixed function $f$ has trivially sample complexity $1$.
In particular, for non-pathological sets $\cX$, combining \Cref{prop:DOO:certified,thm:lower:bound:sample:dependent} with \cref{t:hansen} gives the following near-tight characterization of the optimal sample complexity of certified algorithms.

\begin{theorem}
\label{t:main-thm}
Let $L>0$ and suppose that \cref{ass:geomCx} holds with $r_0 > \nicefrac{\eo}{2L}$.
Then, there exist two constants $c_d,C_d>0$ such that, for all \lip{} functions $f\in\cF_L$ and any accuracy $\e \in (0,\eo]$, 
the optimal sample complexity of any certified algorithm $A$ satisfies
\begin{multline*}
   \textcolor{black}{
\frac{
c_d
\left( 1- \frac{\Lip(f)}{L} \right)^d
}{
 1+   \bce{ \log_2(\eo/\e) } 
}
}
    \int_\cX \frac{\mathrm{d}\bx}{\brb{ f(\bxs) - f(\bx) + \e }^d}
\le
    \inf_A \sigma(A,f,\epsilon) 
\\
\le  
    \sigma(\text{\emph{c.DOO}}, f, \e) 
\le
    \textcolor{black}{C_d}
    \int_\cX \frac{\mathrm{d}\bx}{\brb{ f(\bxs) - f(\bx) + \e }^d}
 \;.
\end{multline*}
For example, one can take $c_d = \frac{L^d}{4\cdot 2^{7d} v_1}$ and $C_d = \left(2+
K
\brb{
\mathbf{1}_{ \nu/ R \geq 1 }
+
\mathbf{1}_{ \nu/ R < 1 }
(
\fracc{4 R}{\nu}
)^d
} \right) \frac{(128L)^d }{ v_1 \gamma}$.
\end{theorem}

\paragraph{The boundary case $L=\Lip(f)$.}

The previous results show that, for any function $f$ whose best Lipschitz constant $\Lip(f)$ is bounded away from $L$, the quantity $\SC(f,\e)$ (or the equivalent integral bound) characterizes the optimal sample complexity up to log terms and dimension-dependent constants. For the sake of completeness, we now discuss the boundary case in which $L=\Lip(f)$, i.e., when the best \lip{} constant of the target function is known exactly by the algorithm.
As we mentioned earlier, this case is not of great relevance in practice, as one can hardly think of a scenario in which $f$ is unknown but the learner has somehow perfect knowledge of its smallest \lip{} constant $\Lip(f)$.
It could, however, be of some theoretical interest.

Next we show an interesting difference between dimensions $d=1$ and $d \geq 2$. In dimension one, $\SC(f,\e)$ nearly characterizes the optimal sample complexity even when $L=\Lip(f)$, as shown below. The proof is deferred to Section~\ref{s:missing-lower-1} of the \suppl{}.

\begin{proposition}
\label{prop:lower:bound:dim:one}
If $d=1$, let $c=2^{-8}/3$.
Then, the sample complexity of any certified algorithm $A$ satisfies, for any $L$-\lip{} function $f$ and all $\e \in (0,\eo]$,
$
\sigma(A,f,\epsilon) 
>
c
S_\mathrm{C}(f,\e) / (1+m_{\epsilon})
$.
\end{proposition}

In contrast, in dimension $d \geq 2$, there are examples of functions $f$ with $\Lip(f)=L$ for which the optimal sample complexity is much smaller than $\SC(f,\e)$. In the proposition below, the fact that $\Lip(f)=L$, the specific shape of $f$, and the specific initialization point enable the learner to find \emph{and certify} a maximizer of $f$ in only two rounds, while $\SC(f,\e)$ grows polynomially in $1/\epsilon$. The proof is deferred to Section~\ref{s:missing-lower-2} of the \suppl{}, together with a refresher on  $\mathrm{PS}$.

\begin{proposition} \label{prop:counter:example:Nepsilon}
Let $d \ge 2$, $\cX := B_1$, and $\lno{\cdot}$ be a norm.
The certified \piya{}-Shubert algorithm $\mathrm{PS}$ with initial guess $\bx_1 := \bzero$ is a certified algorithm satisfying, for the $L$-\lip{} function $f := L \lno{\cdot}$ and any $\e \in (0,\eo \textcolor{black}{)}$, $\sigma(\mathrm{PS}, f, \e) = 2 \ll \textcolor{black}{c_d} / \e^{d-1} \le \SC(f,\e)$, for some constant $\textcolor{black}{c_d}>0$.
\end{proposition}

Note that the same upper bound of $2$ could be proved in dimension $d=1$ for $f(x)=L|x|$, but in this case, $S_\mathrm{C}(f,\e) / (1+m_{\epsilon})$ is of constant order. There is therefore no contradiction with Proposition~\ref{prop:lower:bound:dim:one}.

Also, remark that Proposition~\ref{prop:counter:example:Nepsilon} does not solve the question of characterizing the optimal sample complexity for any $f$ such that $\Lip(f)=L$. We conjecture that the drastic improvement in the sample complexity shown in the specific example above is not possible for all other functions $f$ with $\Lip(f)=L$. For instance, we believe that $S_\mathrm{C}(f,\e)$ remains the right quantity for functions for which $\babs{ f(x)-f(y)}/\lno{ x-y }$ is close to $\Lip(f)$ only far away from the set of maximizers.

\section{Conclusions and Open Problems}
\textbf{Contributions.}
We studied the sample complexity of certified zeroth-order \lip{} optimization.
We first showed that the sample complexity of the computationally tractable $\cdoo$ algorithm scales with the quantity $\SC(f,\e)$ introduced in \cref{eq:certified:bound} (\cref{prop:DOO:certified}).
We then characterized this quantity in terms of the integral expression $\int_{\cX} \mathrm{d}\bx/( \max(f) - f(\bx) + \epsilon )^d$ (\cref{t:hansen}), solving as a corollary a long-standing open problem in \cite{hansen1991number}. Finally we proved an instance-dependent lower bound (Theorem~\ref{thm:lower:bound:sample:dependent}) showing that this integral characterizes (up to $\log$ factors) the optimal sample complexity of certified zeroth-order \lip{} optimization in any dimensions $d \ge 1$ whenever the smallest \lip{} constant of $f$ is not known exactly by the algorithm (\cref{t:main-thm}). 

\textbf{Limitations and open problems.}
There are some interesting directions that would be worth investigating in the future but we did not cover in this paper. First, even if the results of \Cref{s:improvedDOO} could be easily extended to pseudo-metric spaces as in \citet{Mun-14-MonteCarloTreeSearch} and related works, our other results are finite-dimensional and exploit the normed space structure.

Second, our lower and upper bounds involve constants (with respect to $f$ and $\e$) that depend exponentially on $d$. Although the dependency in $d$ was intentionally not optimized here, it would be beneficial to obtain tighter constants. However, we believe that removing the exponential dependency in $d$ altogether, while keeping the same level of generality as this paper, would be very challenging, if at all possible. There is also room for improvement in the specific (and unrealistic) case $L=\Lip(f)$, for which characterizing the optimal sample complexity is still open in dimensions $d \geq 2$.

A third question is related to the general notion of adaptivity to smoothness (e.g., \citealt{munos2011optimistic,BGV19-ParameterFreeOptimization}). Note that when no upper bound $L$ on $\Lip(f)$ is available, it is in general not possible to issue a finite certificate $
\xi_n$ satisfying $\max(f)-f(\bxs_n)\le \xi_n$ for any $f$ (as arbitrarily steep bumps can be added to $f$). 
Hence, while having some upper bound $L$ is necessary for certified optimization, a natural question is whether $L$ could be much larger than $\Lip(f)$ without penalizing significantly the sample complexity. Theorem \ref{thm:lower:bound:sample:dependent} provides a negative answer since, for a fixed $f$ for which $\bxs$ is in the interior of $\cX$, when $L \to \infty$, the lower bound on $\sigma(A,f,\e)$ is of order $S_C(f,\e) / \log(1 / \e)$ and $S_C(f,\e)$ grows like $L^d$. It would thus be interesting to see whether a relaxed version of certification would make sense for cases when only a very coarse upper bound $L$ on $\Lip(f)$ is known. 

Finally, it would also be interesting to see if, using randomized algorithms (still with exact observations of $f$) and weakening the notion of certification (requiring it to hold only with some prescribed probability), smaller upper bounds could be obtained. Also, an important question is to quantify the sample complexity increase yielded by having
only noisy observations of $f$ \citep{bubeck2011x,kleinberg2019bandits}. We expect the modified integral $\int_{\cX} \mathrm{d}\bx/( \max(f) - f(\bx) + \epsilon )^{d+2}$ to play a role for this, where our intuition for the ``$+2$'' is that to see an $\e$-big gap through the variance, the learner has to draw roughly $\e^{-2}$ samples.

\begin{ack}
The work of Tommaso Cesari and S{\'e}bastien Gerchinovitz has benefitted from the AI Interdisciplinary Institute ANITI, which is funded by the French ``Investing for the Future – PIA3'' program under the Grant agreement ANR-19-P3IA-0004. S\'{e}bastien Gerchinovitz gratefully acknowledges the support of the DEEL project\footnote{\url{https://www.deel.ai/}}. This work benefited from the support of the project BOLD from the French national research agency (ANR).
\end{ack}

\bibliographystyle{ACM-Reference-Format}
\bibliography{biblio}

\section*{Checklist}


\begin{enumerate}

\item For all authors...
\begin{enumerate}
  \item Do the main claims made in the abstract and introduction accurately reflect the paper's contributions and scope?
    \answerYes{The outline in \cref{sec:outline} provides pointers to where the claimed contributions of the paper are provided.}
  \item Did you describe the limitations of your work?
   \answerYes{In particular, \textcolor{black}{the conclusion discusses limitations and open problems.}}
     
  \item Did you discuss any potential negative societal impacts of your work?
   \answerNA{This is a theoretical/foundation work that adds to the theory and methodology of optimization. As for any such contributions, the positive or negative societal impact will depend on the application case. We do not promote any harmful use of this theory, but we expand on the existing knowledge.}
   
  \item Have you read the ethics review guidelines and ensured that your paper conforms to them?
    \answerYes{See previous question.}
\end{enumerate}

\item If you are including theoretical results...
\begin{enumerate}
  \item Did you state the full set of assumptions of all theoretical results?
   \answerYes{All our results explicitly refer to their required assumptions. Some general assumptions that hold throughout the paper are also stated in the introduction.} 
	\item Did you include complete proofs of all theoretical results? \answerYes{Some technical details are deferred to the \suppl{}.}
\end{enumerate}

\item If you ran experiments...
\begin{enumerate}
  \item Did you include the code, data, and instructions needed to reproduce the main experimental results (either in the supplemental material or as a URL)?
    \answerNA{We did not ran experiments.}
  \item Did you specify all the training details (e.g., data splits, hyperparameters, how they were chosen)?
    \answerNA{}
	\item Did you report error bars (e.g., with respect to the random seed after running experiments multiple times)?
    \answerNA{}
	\item Did you include the total amount of compute and the type of resources used (e.g., type of GPUs, internal cluster, or cloud provider)?
    \answerNA{}
\end{enumerate}

\item If you are using existing assets (e.g., code, data, models) or curating/releasing new assets...
\begin{enumerate}
  \item If your work uses existing assets, did you cite the creators?
   \answerNA{We did not use existing assets (code, data or models) nor cure/release new assets (code, data or models).}
  \item Did you mention the license of the assets?
    \answerNA{}
  \item Did you include any new assets either in the supplemental material or as a URL?
    \answerNA{} 
  \item Did you discuss whether and how consent was obtained from people whose data you're using/curating?
    \answerNA{}
  \item Did you discuss whether the data you are using/curating contains personally identifiable information or offensive content?
    \answerNA{}
\end{enumerate}

\item If you used crowdsourcing or conducted research with human subjects...
\begin{enumerate}
  \item Did you include the full text of instructions given to participants and screenshots, if applicable?
   \answerNA{We did not use crowdsourcing nor conducted research with human subjects.}
  \item Did you describe any potential participant risks, with links to Institutional Review Board (IRB) approvals, if applicable?
    \answerNA{}
  \item Did you include the estimated hourly wage paid to participants and the total amount spent on participant compensation?
    \answerNA{}
\end{enumerate}

\end{enumerate}


\newpage

\begin{center}

  \vbox{%
    \hsize\textwidth
    \linewidth\hsize
    \vskip 0.1in
          \hrule height 4pt
          \vskip 0.25in
          \vskip -\parskip%
    \centering
    {\LARGE{\bf \papertitle } \\ \smallskip Supplementary Material \par}
          \vskip 0.29in
          \vskip -\parskip
          \hrule height 1pt
          \vskip 0.09in%
      \begin{tabular}[t]{c}\bf\rule{0pt}{24pt}
       Fran\c{c}ois Bachoc\\
 Institut de Math\'ematiques de Toulouse \& University Paul Sabatier \\
  \texttt{francois.bachoc@math.univ-toulouse.fr} \\
~ \\
  {\bf  Tommaso Cesari} \\
   Toulouse School of Economics \\
   \texttt{tommaso-renato.cesari@univ-toulouse.fr} \\
~ \\
   {\bf S\'ebastien Gerchinovitz} \\
   IRT Saint Exup\'ery \& Institut de Math\'ematiques de Toulouse \\
   \texttt{sebastien.gerchinovitz@irt-saintexupery.com}
      \end{tabular}%
    \vskip 0.3in minus 0.1in
  }

\end{center}

\appendix

\section{Useful Results on Packing and Covering}
\label{s:packing-covering}

For the sake of completeness, we recall the definitions of packing and covering numbers, as well as some known useful results.

\begin{definition}
Fix any norm $\lno{\cdot}$.
For any non-empty and bounded subset $E$ of $\Rd$ and all $r>0$, 
\begin{itemize}
    \item the $r$-\emph{packing number} of $E$ is the largest cardinality of an $r$-packing of $E$, i.e.,
    \[
    	\cN(E,r)
    := 
    	\sup \bcb{
    	\textstyle{
    		k \in \Ns  :  \exists \bx_1, \ldots, \bx_k \in E, \min_{i\neq j} \lno{ \bx_i - \bx_j } > r
    	}
    	} \;;
    \]
    \item the $r$-\emph{covering number} of $E$ is the smallest cardinality of an $r$-covering of $E$, i.e.,
    \[
    	\cM(E,r)
    :=
    	\min \bcb{
    		k \in \Ns  :  \exists \bx_1,\ldots,\bx_k \in \Rd, \forall \bx \in E, \exists i\in [ k ], \lno{ \bx-\bx_i}\le r
    	} \;.
    \]
\end{itemize}
We also define $\cN(\varnothing,r)=\cM(\varnothing,r)=0$ for all $r>0$.
\end{definition}

Covering numbers and packing numbers are closely related. 
In particular, the following well-known inequalities hold---see, e.g., \citep[Lemmas~5.5~and~5.7, with permuted notation of $\cM$ and $\cN$]{Wainwright19-HighDimensionalStatistics}.\footnote{%
The definition of $r$-covering number of a subset $E$ of $\R^d$ implied by \citep[Definition 5.1]{Wainwright19-HighDimensionalStatistics} is slightly stronger than the one used in our paper, because elements $x_1, \ldots, x_N$ of $r$-covers belong to $E$ rather than just $\R^d$. Even if we do not need it for our analysis, Inequality~(\ref{eq:upper-convering}) holds also in this stronger sense.}
\begin{lemma}
\label{l:wainwright}
Fix any norm $\lno{\cdot}$.
For any bounded set $E\s\R^d$ and all $r>0$,
\begin{equation}
\label{eq:wainwright}
	\cN(E,2r)
\le
	\cM(E,r)
\le
	\cN(E,r)\;.
\end{equation}
Furthermore, for any $\delta>0$ and all $r>0$,
\begin{equation}
\label{eq:upper-convering}
	\cM\lrb{B_{\delta},r}
\le
	\lrb{1 + 2 \frac{\delta}{r} \I_{r<\delta}}^d\;.
\end{equation}
\end{lemma}
We now state a known lemma about packing numbers at different scales. 
This is the go-to result for rescaling packing numbers. 

\begin{lemma} \label{lem:changing radius}
Fix any norm $\lno{\cdot}$. 
For any bounded $E \subset \Rd$ and all $0 < r_1 < r_2 < \infty$, we have
\[
\cN \lrb{ E, r_1 }
\leq  
\left( 4 \frac{r_2}{r_1} \right)^d
\cN \lrb{ E, r_2 } \;.
\]
\end{lemma}
\begin{proof}
Fix any bounded $E\s \Rd$ and $0 < r_1 < r_2 < \iop$.
Consider an $r_1$-packing $F = \{\bx_1,\ldots,\bx_{N_1}\}$ of $E$ with cardinality $N_1 = \cN \lrb{ E, r_1 }$. Consider then the following iterative procedure. Let $F_0 = F$ and initialize $k=1$. While $F_{k-1}$ is non-empty, let $\by_k$ be any point in $F_{k-1}$, let $F_k$ be obtained from $F_{k-1}$ by removing the points at $\lno \cdot$-distance less or equal to $r_2$ from $\by_k$ (including $\by_k$ itself), and increase $k$ by one. Then this procedure yields an $r_2$-packing of $E$ with cardinality equal to the number of steps (the final value $k_{\mathrm{fin}}$ of $k$). 
At each step $k$, the balls with radius $r_1/2$ centered at points that are removed at this step are included in the ball with radius $2 r_2$ centered at $\by_k$. 
By a volume argument, then, the number of removed points at each step is smaller than or equal to $v_{2 r_2}/v_{r_1/2} =(4 r_2 / r_1)^d $. 
Hence the total number of steps $k_{\mathrm{fin}}$ is greater than or equal to $\cN \lrb{ E, r_1 }  (r_1 / 4 r_2)^d$.  This concludes the proof since $ \cN \lrb{ E, r_2 }$ is greater than or equal to the total number of steps $k_{\mathrm{fin}}$.
\end{proof}

\section{Missing Proofs of Section~\ref{s:improvedDOO}}
\label{s:missing-upper-DOO}

In \cref{s:improvedDOO}, we introduced a certified version of the DOO algorithm. 
To prove \cref{prop:DOO:certified}, we adapt and slightly improve (see \cref{rem:doo:suboptimal:certified}) some of the arguments in \cite{munos2011optimistic}, showing that the sample complexity of $\cdoo{}$ is upper bounded (up to constants) by $S_\mathrm{C}(f,\e)$, defined in \cref{eq:certified:bound}.

\paragraph{Proof of \cref{prop:DOO:certified}.}
Recall that $f$ is an $L$-\lip{} function with a global maximizer $\bxs$.

Let us first show that Algorithm \ref{alg:DOO} is indeed a certified algorithm, that is, $f(\bx^\star )-f(\bx_n) \le \xi_n$ for all $n \geq 1$. Note that $f(\bx^\star )-f(\bx_1) \le LR = \xi_1$, since $f$ is $L$-\lip{} and $R$ bounds the diameter of $X_{0,0}\supset \cX$.
So, take any $n \ge 2$. 
Consider the state of the algorithm after exactly $n$ evaluations of $f$. Let $(h^\star,i^\star)$ correspond to the last time that \cref{eq:DOO:select:cell} was reached and let $m$ be the total number of evaluations of $f$ made up to that time ($m\le n$). Then the error certificate is $\xi_n = f(\bx_{h^\star,i^\star}) +  L R \delta^{h^\star} - 
f(\bxs_n)$. By induction, it is straightforward to show that the union of the cells in $\mathcal{L}_j$ contains $\cX$ at all steps $j \in \mathbb{N}^\star$. Therefore, the global maximizer $\bx^\star$ belongs to a cell $X_{\bar{h},\bar{i}}$ with $(\bar{h},\bar{i}) \in \mathcal{L}_m$. We have, using first \cref{eq:DOO:select:cell} and then Assumption \ref{assumption:DOO:small:cells} and that $f$ is $L$-\lip{},
\begin{align}
    f(\bx_{h^\star,i^\star}) + L R \delta^{h^\star}
    & \geq f(\bx_{\bar{h},\bar{i}}) + L R \delta^{\bar{h}} \nonumber
    \\
    & \geq f(\bx^\star) - L R \delta^{\bar{h}} + L R \delta^{\bar{h}} \nonumber
    \\
    & = f(\bx^\star) \;.\label{eq:DOO:selected:points:goodbis}
\end{align}
This shows $\xi_n \geq f(\bx^\star) - 
f(\bxs_n)$.
Hence Algorithm \ref{alg:DOO} is a certified algorithm.

We now show the upper bound on $\sigma(\cdoo,f,\e)$.
Consider the infinite sequence $((h^\star_\ell,i^\star_\ell))_{\ell \in \Ns}$ of the leaves that are successively selected at \cref{eq:DOO:select:cell} of Algorithm \ref{alg:DOO}.
For any leaf $(h,i) \in ((h^\star_\ell,i^\star_\ell))_{\ell \in \Ns}$, let $N_{h,i}$ be the number of evaluations of $f$ made by Algorithm \ref{alg:DOO} until the leaf $(h,i)$ is selected at \cref{eq:DOO:select:cell}.  
Define then the stopping time
\[
I_{\e} = \inf \left\{ \ell \in \Ns ;  
f( \bx_{h^\star_\ell,i^\star_\ell}  )
+ L R \delta^{h^\star_\ell} 
\leq 
\underset{i \in [N_{h^\star_\ell,i^\star_\ell}]}{\max} f(\bx_i) 
+ \e 
\right\} \;
\]
which corresponds to the first iteration when the event
\begin{equation} \label{eq:the:event}
f( \bx_{h^\star,i^\star}  )
+ L R \delta^{h^\star} 
\leq 
\underset{i \in [N_{h^\star,i^\star}]}{\max} f(\bx_i) 
+ \e
\end{equation}
holds at \cref{eq:DOO:select:cell}. 
Consider the $N$-th evaluation of $f$ with $N = N_{h^\star_{I_{\e}},i^\star_{I_{\e}}} + 1$, that is, the first evaluation of $f$ after the event \eqref{eq:the:event} holds for the first time. Then we have, from \eqref{eq:the:event} with $(h^\star , i^\star) = (h^\star_{I_\e},i^\star_{I_\e})$,
\[
f( \bx_{h^\star,i^\star} ) + L R \delta^{h^\star} 
\leq 
\underset{i \in [N_{h^\star,i^\star}]}{\max} f(\bx_i) 
+\e 
\leq
\underset{i \in [N]}{\max} f(\bx_i) 
+\e, 
\]
and thus $\xi_N \leq \e$. 
Since by definition $\sigma(\cdoo,f,\e) = \min \{n \in \N^*: \xi_n \leq \e\}$, we have
\begin{equation} \label{eq:DOO:certified:tau:smaller}
\sigma(\cdoo,f,\e)
\leq N = 
N_{h^\star_{I_{\e}},i^\star_{I_{\e}}} +1
\leq 2 + K (I_\e-1) \;. 
\end{equation}

We now bound $I_\e-1$ from above. 
Assume without loss of generality that $I_\e - 1 \ge 1$ and
consider the sequence $(h^\star_1,i^\star_1), \ldots , (h^\star_{I_\e-1},i^\star_{I_\e-1})$ corresponding to the first $I_{\e}-1$ times the DOO algorithm went through \cref{eq:DOO:select:cell}. Let $\mathcal{E}_{\e}$ be the corresponding finite set $\{ \bx_{h^\star_1,i^\star_1} , \ldots , \bx_{h^\star_{I_\e-1},i^\star_{I_\e-1}} \}$. 
Recall that $\e_{\me} := \e$ and $\e_i := \e_0 2^{-i}$ for $i < \me $, with $\eo := L \max_{\bx,\by \in \cX} \lno{\bx-\by}$.
Recall also that $\cX_\e := \{\bx \in \cX: \max(f) - f(\bx) \le \e\}$ and for all $0\le a<b$, 
$
    \cX_{(a,b]}
:=
	\bcb{ \bx \in  \cX   :  a < f(\bxs)-f(\bx)\le b}
$.
Since any $\bx \in \cX$ is $\eo$-optimal, then it either belongs to $\cX_\e$ or one of the layers $\cX_{(\e_i, \e_{i-1}]}$.
Thus we have $\mathcal{E}_{\e}
\subset
\cX_{ \e}
\bigcup
\left(
\bigcup_{i=1}^{\me} 
\cX_{\left(\e_i, \ \e_{i-1} \right]}
\right)$, so that
\begin{equation} \label{eq:DOO:certified:E:in:cup}
I_\e -1 = \card(\mathcal{E}_{\e})
\leq \card \left(\mathcal{E}_{\e} \cap \cX_{ \e}\right)
+ \sum_{i=1}^{\me} 
\card \left(\mathcal{E}_{\e} \cap \cX_{\left(\e_i, \ \e_{i-1} \right]}\right) \;.
\end{equation}
Let $N_{\e,m_\e+1}$ be the cardinality of $\mathcal{E}_{\e} \cap \cX_{ \e}$. For $i=1,\ldots,m_\e$, let $N_{\e,i}$ be the cardinality of $\mathcal{E}_{\e} \cap \cX_{\left(\e_i, \ \e_{i-1} \right]}$.

Note that the arguments leading to \eqref{eq:DOO:selected:points:goodbis} 
imply that for $\bx_{h,j} \in \mathcal{E}_{\e}$, 
\begin{equation} \label{eq:DOO:selected:points:good}
    \bx_{h,j} \in \cX_{ L R \delta^{h} }.
\end{equation}
Consider two distinct $\bx_{h,j} , \bx_{h',j'} \in \mathcal{E}_{\e} \cap \cX_{\left(\e_i,
 \ \e_{i-1} \right]}$. Then, from Assumption \ref{assumption:DOO:well-separated} and \eqref{eq:DOO:selected:points:good}, we obtain
\[
|| \bx_{h,j} - \bx_{h',j'} ||
\geq
\nu \delta^{ \max( h , h' ) }
> 
\frac{\nu \e_i}{L R   } \;.
\]
Hence, by definition of packing numbers, we have
\[
N_{\e,i} 
\leq 
\mathcal{N}
\left(
\cX_{\left(\e_i,
 \ \e_{i-1} \right]}
,
\frac{\nu \e_i}{L R   }
\right).
\]
Using now Lemma \ref{lem:changing radius} (Section~\ref{s:packing-covering}), we obtain
\begin{equation} \label{eq:DOO:non:certified:Nei:smaller}
N_{\e,i} 
\leq 
\left(
\mathbf{1}_{ \nu/ R \geq 1 }
+
\mathbf{1}_{ \nu/ R < 1 }
\left(
\frac{4 R}{\nu}
\right)^d
\right)
\cN
\left(
\cX_{\left(\e_i,
 \ \e_{i-1} \right]}
,
\frac{ \e_i }{ L    }
\right).
\end{equation}

Let now $\bx_{h^\star_\ell,i^\star_\ell} \in \mathcal{E}_{\e} \cap \cX_{\e}$, with $\ell \in \{1 , \ldots, I_\e-1 \} $. The leaf $(h^\star_\ell,i^\star_\ell)$ was selected when the algorithm went through \cref{eq:DOO:select:cell} for the $\ell$-th time. By definition of $I_\e$, the event \eqref{eq:the:event} does not hold when $(h^\star , i^\star) = (h^\star_\ell,i^\star_\ell)$ and thus
\[
f(\bx_{h^\star_{\ell},i^\star_{\ell}}) + L R \delta^{h^\star_{\ell}} 
> 
\underset{i \in [N_{h^\star_{\ell} , i^\star_{\ell}}]}{\max} f( \bx_i )
+ \e
\geq 
f(\bx_{h^\star_{\ell},i^\star_{\ell}}) + \e \;.
\]
This implies that $L R \delta^{h^\star_{\ell}} > \e$ and thus
\[
 \delta^{h^\star_{\ell}} > 
\frac{ \e }{L R} \;.
\]
Now consider two distinct $\bx_{h,j} , \bx_{h',j'} \in \mathcal{E}_{\e} \cap \cX_{\e}$. Then, from Assumption \ref{assumption:DOO:well-separated}, we obtain
\[
|| \bx_{h,j} - \bx_{h',j'} ||
\geq 
\nu \delta^{ \max( h , h' ) }
>
\frac{\nu \e}{L R } \;.
\]
Hence, we have
\[
N_{\e,m_\e+1} 
\leq 
\mathcal{N}
\left(
\cX_{\e}
,
\frac{\nu \e}{L R }
\right).
\]
Using now Lemma \ref{lem:changing radius} (Section~\ref{s:packing-covering}), we obtain
\begin{equation*} 
N_{\e,m_\e+1} 
\leq 
\left(
\mathbf{1}_{ \nu/ R \geq 1 }
+
\mathbf{1}_{ \nu/ R < 1 }
\left(
\frac{4 R}{\nu}
\right)^d
\right)
\cN
\left(
\cX_{\e}
,
\frac{ \e }{ L    }
\right).
\end{equation*}
Combining \eqref{eq:DOO:certified:tau:smaller} and \eqref{eq:DOO:certified:E:in:cup} with \eqref{eq:DOO:non:certified:Nei:smaller} 
and the last inequality concludes the proof. \hfill \qedsymbol{}

\begin{remark} \label{rem:doo:suboptimal:certified}
The analysis of the DOO algorithm in \citet[Theorem 1]{munos2011optimistic} does not address the certified setting. The previous proof adapts this analysis to the certified setting and, in passing, slightly improves some of the arguments. Indeed, when counting the cell representatives that are selected, \citet[Theorem~1]{munos2011optimistic} partitions the domain $\cX$ at any depth $h$ of the tree, yielding bounds involving packing numbers of the form $\cN \lrb{ \cX_{\e_{k-1}}, \, \frac{\e_k}{L} }$, $k = 1,\ldots,\me$. In contrast we partition the values of $f$, yielding bounds involving the smaller packing numbers $\cN \lrb{ \cX_{(\e_k, \e_{k-1}]}, \, \frac{\e_k}{L} }$, $k = 1,\ldots,\me$ (and $\cN \lrb{ \cX_{\e}, \frac{\e}{L} }$ that is specific to the certified setting). This improvement also enables us to slightly refine the bound of \citet[Theorem 1]{munos2011optimistic} in the non-certified setting, see Remark \ref{rem:doo:suboptimal:non:certified} in Section~\ref{s:comparison-non-certif}. We also refer to this remark for more details on the two partitions just discussed.
\end{remark}

\begin{remark}
\label{rem:generalparition}
The bound of \cref{prop:DOO:certified}, based on \cref{eq:certified:bound}, is built by partitioning $[0,\e_0]$ into the $\me+1$ sets $[0, \e],(\epsilon,\epsilon_{m_{\e}-1}],(\e_{m_{\e}-1},\e_{m_\e-2}],\ldots,(\e_1,\e_0]$ whose lengths are sequentially doubled (except from
$[0, \e]$ to $(\epsilon,\epsilon_{m_{\e}-1}]$ and from
$(\e,\e_{m_\e-1}]$ to $(\e_{m_{\e}-1},\e_{m_\e-2}]$). As can be seen from the proof of \cref{prop:DOO:certified}, more general bounds could be obtained, based on more general partitions of  $[0,\e_0]$. The benefits of the present partition are the following. First, except for $[0, \e]$, it considers sets whose upper values are no more than twice the lower values, which controls the magnitude of their corresponding packing numbers in \cref{eq:certified:bound}, at scale the lower values. Second, the number of sets in the partition is logarithmic in $1/\e$ which controls the sum in \cref{eq:certified:bound}. 
Finally, the upper bound is then tight up to a logarithmic factor for functions $f \in \cF_L$, as proved in Section \ref{section:lower:bounds}. 
Note also that the same generalization could be applied in the non-certified setting, see Section~\ref{s:comparison-non-certif}.
\end{remark}

\section{Missing Proofs of Section~\ref{s:hansen}}
\label{s:missing-characterization}

We now prove a result on the sum of volumes of overlapping layers that is used in the proof of \cref{t:hansen}.
\begin{lemma}
\label{l:overlap}
If $f$ is $L$-\lip{}, 
fix $\e\in(0,\eo]$ and recall that $m_{\e}:=\bce{\log_{2}(\e_{0}/\e)}$,
$\e_{m_{\e}}:=\e$, and for all $k\le m_{\e}-1$, $\e_{k}=:\e_{0}2^{-k}$.
Then
\[
\frac{\vol\brb{\cX_{2\e}}}{\e^{d}}+\sum_{k=1}^{m}\frac{\vol\lrb{\cX_{\left(\frac{1}{2}\e_{k},\ 2\e_{k-1}\right]}}}{\e_{k-1}^{d}}\le8^{d}\lrb{\frac{\vol(\cX_{\e})}{\e^{d}}+\sum_{i=1}^{m}\frac{\vol\brb{\cX_{(\e_{k},\ \e_{k-1}]}}}{\e_{k-1}^{d}}}\;.
\]
\end{lemma}

\begin{proof}
To avoid clutter, we denote $m_{\e}$ simply by $m$. \textcolor{black}{Assume first that $m\ge 3$. Then, t}he left hand
side can be upper bounded by
\begin{align*}
& 
    \frac{\vol\brb{\cX_{\e}}+\vol\brb{\cX_{(\e_{m}, \ \e_{m-1}]}}+\vol\brb{\cX_{(\e_{m-1}, \ \e_{m-2}]}}}{\e^{d}}
\\
&
\hspace{51.68187pt}
\qquad
+
    \sum_{k=1}^{m-2}\frac{\vol\lrb{\cX_{\left(\e_{k+1},\ \e_{k}\right]}}+\vol\lrb{\cX_{\left(\e_{k},\ \e_{k-1}\right]}}\vol\lrb{\cX_{\left(\e_{k-1},\ \e_{k-2}\right]}}}{\e_{k-1}^{d}}
\\
& 
\hspace{51.68187pt}
\qquad
+
    \frac{\vol\lrb{\cX_{\e}}+\vol\brb{\cX_{(\e_{m}, \ \e_{m-1}]}}+\vol\brb{\cX_{(\e_{m-1}, \ \e_{m-2}]}}+\vol\brb{\cX_{(\e_{m-2}, \ \e_{m-3}]}}}{\e_{m-2}^{d}}
\\
&
\hspace{51.68187pt}
\qquad
+
    \frac{\vol\lrb{\cX_{\e}}+\vol\brb{\cX_{(\e_{m}, \ \e_{m-1}]}}+\vol\brb{\cX_{(\e_{m-1}, \ \e_{m-2}]}}}{\e_{m-1}^{d}}
\\
& 
\qquad
\le
    3\frac{\vol\brb{\cX_{\e}}}{\e^{d}}+(2^{d}+2)\frac{\vol\brb{\cX_{(\e_{m}, \ \e_{m-1}]}}}{\e_{m-1}^{d}}+(4^{d}+2^{d}+1)\frac{\vol\brb{\cX_{(\e_{m-1}, \ \e_{m-2}]}}}{\e_{m-2}^{d}}
\\
& 
\hspace{51.68187pt}
\qquad
+
\frac{1}{2^{d}}\sum_{k=2}^{m-1}\frac{\vol\lrb{\cX_{\left(\e_{k},\ \e_{k-1}\right]}}}{\e_{k-1}^{d}}+\sum_{k=1}^{m-2}\frac{\vol\lrb{\cX_{\left(\e_{k},\ \e_{k-1}\right]}}}{\e_{k-1}^{d}}+2^{d}\sum_{k=1}^{m-3}\frac{\vol\lrb{\cX_{\left(\e_{k},\ \e_{k-1}\right]}}}{\e_{k-1}^{d}}
\\
& 
\qquad
=
    3\frac{\vol\brb{\cX_{\e}}}{\e^{d}}+\frac{\vol\brb{\cX_{(\e_{m}, \ \e_{m-1}]}}}{\e_{m-1}^{d}}+4^{d}\frac{\vol\brb{\cX_{(\e_{m-1}, \ \e_{m-2}]}}}{\e_{m-2}^{d}}
\\
&
\hspace{51.68187pt}
\qquad
+
    \frac{1}{2^{d}}\sum_{k=2}^{m-1}\frac{\vol\lrb{\cX_{\left(\e_{k},\ \e_{k-1}\right]}}}{\e_{k-1}^{d}}+\sum_{k=1}^{m}\frac{\vol\lrb{\cX_{\left(\e_{k},\ \e_{k-1}\right]}}}{\e_{k-1}^{d}}+2^{d}\sum_{k=1}^{m}\frac{\vol\lrb{\cX_{\left(\e_{k},\ \e_{k-1}\right]}}}{\e_{k-1}^{d}}
\end{align*}
where we applied several times the definition of the $\e_{k}$'s, the
inequality follows by $\nicefrac{1}{\e^{d}}+\nicefrac{1}{\e_{m-1}^{d}}+\nicefrac{1}{\e_{m-2}^{d}}\le\min\bcb{3(\nicefrac{1}{\e^{d}}),\ (2^{d}+2)(\nicefrac{1}{\e_{m-1}^{d}}),\ (4^{d}+2^{d}+1)(\nicefrac{1}{\e_{m-2}^{d}})}$,
and the \textcolor{black}{bound follows after} observing that $\max(3,1,4^{d})=4^{d}$ and $4^{d} + \nicefrac{1}{2^{d}}+1+2^{d}\le8^{d}$.
\textcolor{black}{The simple cases $m=1$ and $m=2$ can be treated similarly.}
\end{proof}

We denote by $A+B$ the \mink{} sum of two sets $A,B$ and
for any set $A$ and all $\lambda \in \R$, we let $\lambda A := \{ \lambda \ba : \ba \in A \}$.

\begin{proposition}
\label{p:pack-vol}
If $f$ is $L$-\lip{} and $\cX$ satisfies Assumption~\ref{ass:geomCx} with $r_0>0,\gamma\in(0,1]$, then, for all $0<w<u< 2L r_0$,
\[
    \cN\lrb{\cX_u, \ \frac{u}{L}}
\le
    \frac{1}{\gamma}
    \frac{\vol\brb{ \cX_{(\nicefrac{3}{2})u} } }{ \vol\lrb{ \frac{u}{2L} B_1 } }
\qquad \text{and} \qquad
    \cN\lrb{\cX_{(w, u]}, \ \frac{w}{L}}
\le
    \frac{1}{\gamma}
    \frac{\vol\brb{ \cX_{ ( \nicefrac{w}{2}, \nicefrac{3u}{2} ]} } }{ \vol\lrb{ \frac{w}{2L} B_1 } } \;.
\]
\end{proposition}

\begin{proof}
Fix any $u>w>0$. 
Let $\eta_1:=\frac{u}{L}$, $\eta_2:=\frac{w}{L}$, $E_1:=\cX$, $E_2:=\cX_w^c$, and $i\in[2]$.
\textcolor{black}{
Note that for any $\eta>0$ and $A \s \cX$, the balls of radius $\nicefrac{\eta}{2}$ centered at the elements of an $\eta$-packing of $A$ intersected with $\cX$ are all disjoint and included in $\brb{ A+B_{\nicefrac{\eta}{2}}(\bzero) } \cap \cX$.
Thus, letting $P_i$ be a set of $\eta_i$-separated points included in $A_i := \cX_u\cap E_i$ with cardinality $\labs {P_i} = \cN(A_i, \eta_i)$, we have
\[
    \vol\Brb{ \brb{ A_i + B_{\nicefrac{\eta_i}{2}} ( \bzero ) } \cap \cX }
\ge
    \sum_{\bx \in P_i}
    \vol\brb{ B_{\nicefrac{\eta_i}{2}} (\bx) \cap \cX }
\ge
    \gamma \vol\brb{ B_{\nicefrac{\eta_i}{2}} (\bzero) } \cN\lrb{ A_i , \eta_i }
    \;,
\]
where the second inequality follows by Assumption~\ref{ass:geomCx}.
We now further upper bound the left-hand side.} 
Take an arbitrary point $\bx_i \in ( \cX_u \cap E_i + B_{\nicefrac{\eta_i}{2}} ) \cap \cX$. 
By definition of \mink{} sum, there exists $\bx_i' \in \cX_u \cap E_i$ such that $\lno{\bx_i - \bx'_i } \le \nicefrac{\eta_i}{2}$.
Hence 
$
    f(\bxs) - f(\bx_i) 
\le 
    f(\bxs) - f(\bx_i') + \babs{ f(\bx_i') - f(\bx_i) } 
\le 
    u + L (\nicefrac{\eta_i}{2}) 
\le 
    (\nicefrac{3}{2}) u$.
This implies that $\bx_i \in \cX_{(\nicefrac{3}{2})u}$, which proves the first inequality. 
For the second one, note that $\bx_2$ satisfies 
$
    f(\bxs) - f(\bx_2) 
\ge 
    f(\bxs) - f(\bx_2') - \babs{ f(\bx_2') - f(\bx_2) } 
\ge 
    w - L (\nicefrac{\eta_2}{2}) 
=
    (\nicefrac{1}{2}) w$.
\end{proof}

\section{Missing Proofs of Section~\ref{section:lower:bounds}}
\label{s:missing-lower}

In this section we provide all missing details and proofs from \Cref{section:lower:bounds}.

\subsection{Missing details in the Proof of Theorem~\ref{thm:lower:bound:sample:dependent}}
\label{s:missing-details-lower}

We claimed that the quantity $\tau$ introduced in \cref{eq:def:tau} lower bounds the sample complexity $\sigma(A,f,\e)$ of any certified algorithm $A$.
To prove this formally, fix an arbitrary certified algorithm $A$, let $N = \sigma(A,f,\e)$, and assume by contradiction
that $N < \tau$. Then we have $\err_N(A) \ge \inf_{A'} \brb{ \err_N(A') }> \e$ by definition of $\tau$. 
This means that there exists an $L$-\lip{} function $g$, coinciding with $f$ on $\bx_1, \ldots , \bx_N$ and such that $\max(g) - \bxs_N > \e$. 
Now, since $\bx_i$, $\bxs_i$, and $\xi_i$ are deterministic functions of the previous observations $f(\bx_1) = g(\bx_1), \ldots, f(\bx_{i-1})=g(\bx_{i-1})$ (for all $i=1,\ldots,N$), running $A$ on either $f$ or $g$ returns the same $\bx_i$, $\bxs_i$, and $\xi_i$ (for all $i=1,\ldots,N$).
Thus we have that $\sigma(A,g,\e) = \sigma(A,f,\e) = N$.
This, together with the fact that $A$ is a certified algorithm, implies that
$\e < \max(g) - \bxs_N \le \xi_N \le \e$, which yields a contradiction. \hfill \qedsymbol{}

\subsection{Proof of Proposition~\ref{prop:lower:bound:dim:one}}
\label{s:missing-lower-1}

Let $f$ be an arbitrary $L$-\lip{} function.
Let ${Q}=8$.
As for the proof of Theorem \ref{thm:lower:bound:sample:dependent}, it is sufficient to show that $\tau > c
S_\mathrm{C}(f,\e) / (1+m_{\epsilon})$, with $\tau$ defined in \eqref{eq:def:tau}.
If $c S_\mathrm{C}(f,\e) / (1+m_{\epsilon}) < 1$, then the result follows by $\tau \geq 1$. Consider then from now on that $c S_\mathrm{C}(f,\e) / (1+m_{\epsilon}) \geq 1$.

Defining $\tilde{\epsilon}$ as in the proof of \cref{thm:lower:bound:sample:dependent}, one can prove similarly that 
$c S_\mathrm{C}(f,\e)/(1+ m_{\epsilon}) \leq 
c \cN \lrb{ \cX_{ \tilde{\epsilon} }, \ \tilde{\epsilon}/2L }$.
From Lemma \ref{lem:changing radius},
\[
\cN \lrb{ \cX_{ \tilde{\epsilon} }, \ \frac{{Q} \tilde{\epsilon}}{L} }
\geq  \frac{1}{8{Q}}   \cN \lrb{ \cX_{ \tilde{\epsilon} }, \ \frac{\tilde{\epsilon}}{2L} }
\geq 
 \frac{1}{8{Q}} 
 \frac{S_\mathrm{C}(f,\e)}{m_{\epsilon} +1} \geq 12 \;,
\]
because $c = 1/96{Q}$ and $c S_\mathrm{C}(f,\e) / (1+m_{\epsilon}) \geq 1$.
Let now $n \leq c S_\mathrm{C}(f,\e) /(1+ m_{\epsilon})$.
Then we have
$n \leq c (8{Q}) \cN \lrb{ \cX_{ \tilde{\epsilon} }, \ {Q} \tilde{\epsilon}/L }$.
Thus, by $c (8{Q}) = 1/12$, $n \leq  \cN \lrb{ \cX_{ \tilde{\epsilon} }, \ {Q} \tilde{\epsilon}/L } /12$, and $\cN \lrb{ \cX_{ \tilde{\epsilon} }, \ {Q} \tilde{\epsilon}/L } \geq 12$, we have
\begin{equation} \label{eq:proof:lower:bound:n:leq:dim:one}
n \leq  
\left\lfloor
\frac{\cN \lrb{ \cX_{ \tilde{\epsilon} }, \ \frac{{Q} \tilde{\epsilon}}{L} }}{2}
\right\rfloor
- 4 \;.
\end{equation}

Consider a certified algorithm $A$ for $L$-\lip{} functions.
Let us consider a ${Q} \tilde{\epsilon} / L$ packing $\tilde x_1 < \tilde x_2 < \dots < \tilde x_N$ of $\mathcal{X}_{\tilde{\epsilon}}$ with $N = \cN \lrb{ \cX_{ \tilde{\epsilon}, \ {Q} \tilde{\epsilon}/L }}$.
Consider the $\lfl{ N/2 } - 1$ disjoint open segments $(\tilde x_1,\tilde x_3)$, $(\tilde x_3,\tilde x_5),$ $\ldots$, 
$(\tilde x_{2\lfloor N/2 \rfloor -3 }$, 
$\tilde x_{2\lfloor N/2 \rfloor -1})$.
Then from \eqref{eq:proof:lower:bound:n:leq:dim:one} there exists $i \in \bcb{1, 3, \ldots, 2 \lfl{ N/2 } - 3 }$ such that the segment $(\tilde x_i,\tilde x_{i+2})$ does not contain any of the points $x_1=x_1(A,f),\ldots,x_n=x_n(A,f)$ that $A$ queries when run on $f$.
Assume that $\tilde x_{i+1} - \tilde x_i \leq \tilde x_{i+2} - \tilde x_{i+1}$ (the case $\tilde x_{i+1} - \tilde x_i > \tilde x_{i+2} - \tilde x_{i+1}$ can be treated analogously; we omit these straightforward details for the sake of conciseness).
Consider the function $h_{+,\tilde{\epsilon}} \colon \cX \to \mathbb{R}$ defined by
{\small
\[
 h_{+,\tilde{\epsilon}}(x)
 =
\begin{cases} 
f(x) & ~ \text{if} ~ x \in \cX \backslash [\tilde x_i,\tilde x_{i+2}] \\
f(\tilde x_i) + L(x - \tilde x_i)  & ~ \text{if} ~ x \in \cX \cap [\tilde x_i , \tilde x_{i+1}] \\ 
f(\tilde x_i) + L(\tilde x_{i+1} - \tilde x_i)
+ (x - \tilde x_{i+1})
\frac{ 
f(\tilde x_{i+2}) - f(\tilde x_i) - L(\tilde x_{i+1} - \tilde x_i)
}{ \tilde x_{i+2} - \tilde x_{i+1} }
& ~ \text{if} ~ x \in \cX \cap (\tilde x_{i+1} , \tilde x_{i+2}] \;.
\end{cases}
\]
}%
We see that $h_{+,\tilde{\epsilon}}$ is $L$-Lipschitz (since $\tilde x_{i+1} - \tilde x_i \leq \tilde x_{i+2} - \tilde x_{i+1}$).
Furthermore, $h_{+,\tilde{\epsilon}}$ coincides with $f$ at all query points $x_1,\ldots,x_n$.
Similarly, consider the function $h_{-,\tilde{\epsilon}} \colon \cX \to \mathbb{R}$ defined by
{
\small
\[
 h_{-,\tilde{\epsilon}}(x)
 =
\begin{cases} 
f(x) & ~ \text{if} ~ x \in \cX \backslash [\tilde x_i,\tilde x_{i+2}] \\
f(\tilde x_i) - L(x - \tilde x_i)  & ~ \text{if} ~ x \in \cX \cap [\tilde x_i , \tilde x_{i+1}] \\ 
f(\tilde x_i) - L(\tilde x_{i+1} - \tilde x_i)
+ (x - \tilde x_{i+1})
\frac{ 
f(\tilde x_{i+2}) - f(\tilde x_i) + L(\tilde x_{i+1} - \tilde x_i)
}{ \tilde x_{i+2} - \tilde x_{i+1} }
& ~ \text{if} ~ x \in \cX \cap (\tilde x_{i+1} , \tilde x_{i+2}] \;.
\end{cases}
\]
}%
As before, $h_{-,\tilde{\epsilon}}$ is $L$-Lipschitz and coincides with $f$ on $x_1,\ldots,x_n$.

Let $x^\star_n = x^\star_n(A,f)$ be the recommendation of $A$ at round $n$ when run on $f$.

\underline{Case 1}: $x^\star_n \in \cX \backslash [\tilde x_i,\tilde x_{i+2}]$. Then, since $\tilde x_{i} \in \mathcal{X}_{\tilde{\epsilon}}$ and $\tilde x_{i+1} - \tilde x_i \geq {Q} \tilde{\epsilon} / L$, we have
\begin{align*}
h_{+,\tilde{\epsilon}}( \tilde x_{i+1} )
-
h_{+,\tilde{\epsilon}}( x^\star_n ) 
 =
f(\tilde x_i) + L(\tilde x_{i+1}-\tilde x_i)
- f( x^\star_n )
 \geq
- \tilde{\epsilon}
+ L \frac{{Q} \tilde{\epsilon}}{L}
=  7 \tilde{\epsilon} \;.
\end{align*}

\underline{Case 2}: $x^\star_n \in \cX \cap \bsb{ \tilde x_i,\fracc{(\tilde x_i+\tilde x_{i+1})}{2} }$. Then, since $\tilde x_{i+1} - \tilde x_i \geq {Q} \tilde{\epsilon} / L$, we have
\begin{align*}
h_{+,\tilde{\epsilon}}( \tilde x_{i+1} )
-
h_{+,\tilde{\epsilon}}( x^\star_n ) 
& =
f(\tilde x_i) + L(\tilde x_{i+1}-\tilde x_i)
- 
f(\tilde x_i)
-
L
(x^\star_n - \tilde x_i)
\geq
 L \frac{\tilde x_{i+1}-\tilde x_i}{2}
 \geq
 4 \tilde{\epsilon} \;.
\end{align*}

\underline{Case 3}: $x^\star_n \in \cX \cap \bsb{ (\tilde x_i+\tilde x_{i+1})/2,\tilde x_{i+1} }$. Then, since $\tilde x_{i+1} - \tilde x_i \geq {Q} \tilde{\epsilon} / L$, we have
\begin{align*}
h_{-,\tilde{\epsilon}}( \tilde x_{i} )
-
h_{-,\tilde{\epsilon}}( x^\star_n ) 
& =
f(\tilde x_i) - 
f(\tilde x_i)
+
L(x^\star_n - \tilde x_i)
 \geq
 L \frac{\tilde x_{i+1}-\tilde x_i}{2}
 \geq
4 
 \tilde{\epsilon} \;.
\end{align*}

\underline{Case 4}: $x^\star_n \in \cX \cap \bsb{ \tilde x_{i+1},(\tilde x_{i+1}+\tilde x_{i+2})/2 }$. Then, since $\tilde x_{i+1} - \tilde x_i \geq {Q} \tilde{\epsilon} / L$,
since $\tilde x_i,\tilde x_{i+2} \in \mathcal{X}_{\tilde{\epsilon}}$,
and since $h_{-,\tilde{\epsilon}}$ is linear increasing on $[\tilde x_{i+1},\tilde x_{i+2}]$ with left value $f(\tilde x_i) - L(\tilde x_{i+1}-\tilde x_i)$ and right value $f(\tilde x_{i+2})$,
we have
\begin{align*}
h_{-,\tilde{\epsilon}}( \tilde x_{i} )
-
h_{-,\tilde{\epsilon}}( x^\star_n ) 
& \geq 
f(\tilde x_i) - 
\frac{
f(\tilde x_i) - L(\tilde x_{i+1}-\tilde x_i) + f(\tilde x_{i+2})
}{
2
}
\\
& =
\frac{ f(\tilde x_i) - f(\tilde x_{i+2})}{2}
+ L \frac{\tilde x_{i+1}-\tilde x_i}{2}
 \geq
-\frac{\tilde{\epsilon}}{2}
+
\frac{{Q}}{2} \tilde{\epsilon}
 \geq  3 \tilde{\epsilon} \;.
\end{align*}

\underline{Case 5}: $x^\star_n \in \cX \cap \bsb{ (\tilde x_{i+1}+\tilde x_{i+2})/2,\tilde x_{i+2} }$. Then, since $\tilde x_{i+1} - \tilde x_i \geq {Q} \tilde{\epsilon} / L$,
since $\tilde x_i,\tilde x_{i+2} \in \mathcal{X}_{\tilde{\epsilon}}$
and since $h_{+,\tilde{\epsilon}}$ is linear decreasing on $[\tilde x_{i+1},\tilde x_{i+2}]$ with left value $f(\tilde x_i) + L(\tilde x_{i+1}-\tilde x_i)$ and right value $f(\tilde x_{i+2})$, we have
\begin{align*}
h_{+,\tilde{\epsilon}}( \tilde x_{i+1} )
-
h_{+,\tilde{\epsilon}}( x^\star_n ) 
& \geq 
f(\tilde x_i) + L(\tilde x_{i+1}-\tilde x_i)
- 
\frac{
f(\tilde x_i) + L(\tilde x_{i+1}-\tilde x_i)
+
f(\tilde x_{i+2})
}{
2
}
\\
& =
\frac{ f(\tilde x_i) - f(\tilde x_{i+2})}{2}
+ L \frac{\tilde x_{i+1}-\tilde x_i}{2}
 \geq
-\frac{\tilde{\epsilon}}{2}
+
\frac{{Q}}{2} \tilde{\epsilon}
 \geq  3 \tilde{\epsilon} \;.
\end{align*}

Putting all cases together and recalling the definition of $\err_n(A)$ in the proof of \cref{thm:lower:bound:sample:dependent}, we then obtain $\err_n(A)  \ge 3 \tilde{\epsilon} > \epsilon $. 
Being $A$ arbitrary, this implies $\inf_{A'} \err_n(A')  > \epsilon $.
Since this has been shown for any $n \leq c S_\mathrm{C}(f,\e) /(1+ m_{\epsilon})$ we thus have $\tau 
> 
c S_\mathrm{C}(f,\e) /(1+m_{\epsilon})$. \hfill\qedsymbol{}

\subsection{The \piya{}-Shubert Algorithm and Proof of Proposition~\ref{prop:counter:example:Nepsilon}}
\label{s:missing-lower-2}

\paragraph{The \piya{}-Shubert Algorithm.}
In this section, we recall the definition of the certified \piya{}-Shubert algorithm (\cref{alg:piya}, \citealt{Piy-72-AbsoluteExtremum,shubert1972sequential}) and we show that if $\Lip(f)=L$ (i.e., if the best \lip{} constant of $f$ is known exactly by the algorithm) the sample complexity can be constant in dimension $d\ge 2$ (\cref{prop:counter:example:Nepsilon}).

\begin{algorithm2e}
\DontPrintSemicolon
\SetAlgoNoLine
\SetAlgoNoEnd
\SetKwInput{kwIn}{input}
\kwIn{\lip{} constant $L>0$, norm $\lno{\cdot}$, initial guess $\bx_1 \in \cX$}
\For
{%
    $i = 1, 2, \ldots$
}
{
    pick the next query point $\bx_i$\;
    observe the value $f(\bx_i)$\;
    output the recommendation $\bxs_i \gets \argmax_{\bx\in\{\bx_1, \ldots, \bx_i\}} f(\bx) $\;
    output the error certificate $\xi_i = \widehat f^\star_{i} - f^\star_{i}  $, where $\widehat f_i(\cdot) \gets \min_{j\in[i]} \bcb{ f(\bx_j) + L \lno{ \bx_j - (\cdot) } }$, 
    $\widehat f ^\star_i \gets \max_{\bx \in \cX} \widehat f_i(\bx)$,
    $f^\star_i \gets \max_{j\in[i]} f(\bx_j)$, 
    and let $\bx_{i+1} \in \argmax_{\bx \in \cX} \widehat f_i(\bx)$\;
}
\caption{\label{alg:piya} Certified \piya{}-Shubert algorithm (PS)}
\end{algorithm2e}

\paragraph{Proof of \cref{prop:counter:example:Nepsilon}.}
Fix any $\e \in (0,\eo \textcolor{black}{)}$ and any $L$-\lip{} function $f$.
Since $f$ is $L$-\lip{}, then $\max_{\bx \in \cX} \widehat f_i(\bx) \ge \max_{\bx \in \cX} f (\bx)$ for all $i \in \Ns$. 
Hence $\max_{\bx\in\cX} f(x) - f(\bxs_i) \le \max_{\bx\in\cX} \widehat f_i(\bx) - f^\star_i = \xi_i$.
This shows that the certified \piya{}-Shubert algorithm is indeed a certified algorithm.
Then, if $f:=L\lno \cdot$ and $\bx_1 := \bzero$, we have that $\widehat f_1 = f$, $\xi_1 = L$, and $\bx_2$ belongs to the the unit sphere, i.e., $\bx_2$ is a maximizer of $f$.
Since $\widehat f_2 = f$, we have that $\xi_2 = 0$, hence $\sigma(\mathrm{PS}, f, \e) = 2$.
Finally, by definition \eqref{eq:certified:bound}, we have $S_\mathrm{C}(f,\e) \geq \cN( \mathrm{argmax}_{\mathcal{X}} f , \e/L ) $. 
Since $\mathrm{argmax}_{\mathcal{X}} f$ is the unit sphere, there exists a constant $\textcolor{black}{c_d}$, only depending on $d$, $\lno{\cdot}$ and $L$, such that $\SC(f,\e) \geq \textcolor{black}{c_d} / \e^{d-1}$. \hfill\qedsymbol{}

We give some intuition on \cref{prop:counter:example:Nepsilon}. 
Consider a function $f$ that has Lipschitz constant exactly $L$, and a pair of points in $\mathcal{X}$ whose respective values of $f$ are maximally distant, that is the difference of values of $f$ is exactly $L$ times the norm of the input difference.
This configuration provides strong information on the value of the global maximum of $f$, as is illustrated in the proof of Proposition \ref{prop:counter:example:Nepsilon}. Another interpretation is that when $f$ has Lipschitz constant exactly $L$, there is less flexibility for the $L$-Lipschitz function $g$ that yields the maximal optimization error in $\err_n(A)$ (introduced in the proof of \Cref{thm:lower:bound:sample:dependent}).

\section{Comparison with the classical non-certified setting}
\label{s:comparison-non-certif}

For the interested reader who is not familiar with DOO, in this section, we recall and analyze the classical non-certified version of this algorithm.
As mentioned in \cref{rem:doo:suboptimal:certified}, our analysis is slightly tighter than that of \cite{munos2011optimistic}, and serves as a better comparison for highlighting the differences between the certified and the non-certified settings (see Remark~\ref{rmk:comparison} below).

The difference between our certified version $\cdoo$ and the classical non-certified DOO algorithm (denoted by $\ncdoo$ below) is that the latter does not output any certificates $\xi_1,\xi_2,\ldots$.
In other words, $\ncdoo$ coincides with \cref{alg:DOO} except for  \cref{state:firstCertif,state:doocertificate}.
In particular, it outputs the same query points $\bx_1,\bx_2,\ldots$ and recommendations $\bx_1^\star,\bx_2^\star,\ldots$ as $\cdoo$. The performance of this non-certified algorithm is classically measured by the non-certified sample complexity \eqref{eq:defzeta}, i.e., the smallest number of queries needed before outputting an $\e$-optimal recommendation.

\begin{proposition} \label{prop:DOO:non:certified}
If \cref{assumption:DOO:small:cells,assumption:DOO:well-separated} hold, the non-certified sample complexity of the non-certified DOO algorithm $\ncdoo$
satisfies, for all \lip{} functions $f \in \cF_L$\footnote{Our proof can be easily adapted to the weaker assumption that $f$ is only $L$-\lip{} around a maximizer.} and any accuracy $\e\in (0,\eo]$,
\[
\zeta(\ncdoo,f,\e)
\leq
1 + {C_d} \sum_{k=1}^{\me} \cN \lrb{ \cX_{(\e_k, \e_{k-1}]}, \, \frac{\e_k}{L} } \;, 
\]
where ${C_d} = K
\brb{
\mathbf{1}_{ \nu/ R \geq 1 }
+
\mathbf{1}_{ \nu/ R < 1 }
(
\nicefrac{4 R}{\nu}
)^d
}$.
\end{proposition}

\begin{proof}
The proof of \cref{prop:DOO:certified} (Section~\ref{s:missing-upper-DOO}), from the beginning to \eqref{eq:DOO:selected:points:goodbis}, implies that,
for any $(h^\star,i^\star)$ in \cref{eq:DOO:select:cell} of \cref{alg:DOO},
\begin{equation} \label{eq:DOO:selected:points:good:non:certified}
    f(\bx_{h^\star,i^\star}) \in \cX_{ L R \delta^{h^\star} }.
\end{equation}
The guarantee \eqref{eq:DOO:selected:points:good:non:certified} is classical (e.g., \citealt{munos2011optimistic}). 

We now proceed in a direction that is slightly different from the proof of \citet[Theorem~1]{munos2011optimistic}. Consider the first time at which the DOO algorithm reaches \cref{eq:DOO:select:cell} with $f(\bx_{h^\star,i^\star}) \geq f(\bx^\star) - \epsilon$. Then let $I_{\e}$ be the number of times the DOO algorithm went through \cref{eq:DOO:select:cell} strictly before that time, and denote by $n_{\e}$ the total number of evaluations of $f$ strictly before that same time. We have
\[
n_{\e} \leq 
1 + K I_{\e} \;. 
\]
Furthermore, after $n_\e$ evaluations of $f$, we have, by definitions of the recommendation $\bxs_{n_\e}$ and $n_\e$,
\[
f(\bxs_{n_\e}) = \max_{\bx \in \{\bx_1,\ldots,\bx_{n_\e}\}} f(\bx) \geq f(\bx_{h^\star,i^\star}) \geq f(\bxs)-\e \;.
\]
This inequality entails that the non-certified sample complexity of $\ncdoo$ is bounded by $n_{\e}$ and thus
\begin{equation} \label{eq:DOO:zeta:smaller:non:certified}
\zeta( \ncdoo,f,\e)
\leq 
1 + K I_{\e}.
\end{equation}

We now bound $I_\e$ from above, and assume without loss of generality that $I_\e \ge 1$. Consider now the sequence $(h^\star_1,i^\star_1), \ldots , (h^\star_{I_\e},i^\star_{I_\e})$ corresponding to the first $I_{\e}$ times the DOO algorithm $\ncdoo$ went through \cref{eq:DOO:select:cell}. Let $\mathcal{E}_{\e}$ be the corresponding finite set $\{ \bx_{h^\star_1,i^\star_1} , \ldots , \bx_{h^\star_{I_\e},i^\star_{I_\e}} \}$ (a leaf can never be selected twice). By definition of $I_\e$, we have $\mathcal{E}_{\e} \subseteq \cX_{(\e,\e_0]}$. Since $\epsilon = \epsilon_{m_\e} \leq \epsilon_{m_\e-1} \leq \ldots \leq \epsilon_0$, we have $\mathcal{E}_{\e}
\subseteq
\bigcup_{i=1}^{\me} 
\cX_{\left(\e_i, \ \e_{i-1} \right]}$, so that the cardinality $I_\e$ of $\mathcal{E}_{\e}$ satisfies
\begin{equation} \label{eq:DOO:E:in:cup:non:certified}
I_\e = \card(\mathcal{E}_{\e}) \leq \sum_{i=1}^{\me} 
\card \left(\mathcal{E}_{\e} \cap \cX_{\left(\e_i, \ \e_{i-1} \right]}\right) \;.
\end{equation}
Let $N_{\e,i}$ be the cardinality of $\mathcal{E}_{\e} \cap \cX_{\left(\e_i, \ \e_{i-1} \right]}$.
The same arguments as from \eqref{eq:DOO:certified:E:in:cup} to \eqref{eq:DOO:non:certified:Nei:smaller} in the proof of \cref{prop:DOO:certified} yield
\begin{equation} \label{eq:DOO:non:certified:Nei:smaller:non:certified}
N_{\e,i} 
\leq 
\left(
\mathbf{1}_{ \nu/ R \geq 1 }
+
\mathbf{1}_{ \nu/ R < 1 }
\left(
\frac{4 R}{\nu}
\right)^d
\right)
\cN
\left(
\cX_{\left(\e_i,
 \ \e_{i-1} \right]}
,
\frac{ \e_i }{ L    }
\right).
\end{equation}
Combining the last inequality with \eqref{eq:DOO:zeta:smaller:non:certified} and \eqref{eq:DOO:E:in:cup:non:certified} concludes the proof.
\end{proof}

\begin{remark} \label{rem:doo:suboptimal:non:certified}
The analysis of the DOO algorithm in \citet[Theorem 1]{munos2011optimistic} (non-certified version) yields a bound on the non-certified sample complexity \eqref{eq:defzeta} than can be expressed in the form $1+C \sum_{k=1}^{\me} \cN \lrb{ \cX_{\e_{k-1}}, \, \frac{\e_k}{L} }$, with a constant $C$. The corresponding proof relies on two main arguments. First, when a cell of the form $(h^\star,i^\star)$, $i^\star \in \{0,\ldots,K^{h^\star}-1\}$, is selected in \cref{eq:DOO:select:cell} of \cref{alg:DOO}, then the corresponding cell \representative{} $\bx_{i^\star,h^\star}$ is $LR \delta^{h^\star}$-optimal (we also use this argument).
Second, as a consequence, for a given fixed value of $h^\star$, for the sequence of values of $i^\star$ that are selected in \cref{eq:DOO:select:cell} of \cref{alg:DOO}, the corresponding cell \representative{}s $\bx_{h^\star,i^\star}$ form a packing of $\cX_{L R \delta^{h^\star}}$. 

Our slight refinement in the proof of \cref{prop:DOO:non:certified} stems from the observation that using a packing of $\cX_{ L R \delta^{h^\star}}$ yields a suboptimal analysis, since the cell \representative{}s $\bx_{h^\star,i^\star}$ can be much better than $L R \delta^{h^\star}$-optimal. Hence, we proceed differently from \cite{munos2011optimistic}, by first partitioning all the selected cell \representative{}s (in \cref{eq:DOO:select:cell} of \cref{alg:DOO}) according to their level of optimality as in \eqref{eq:DOO:E:in:cup:non:certified} and then by exhibiting packings of the different layers of input points $\cX_{(\epsilon,\epsilon_{m_{\e}-1}]},\cX_{(\e_{m_{\e}-1},\e_{m_\e-2}]},\ldots,\cX_{(\e_1,\e_0]}$. In a word, we partition the values of $f$ instead of partitioning the input space when counting the representatives selected at all levels.
\end{remark}

\begin{remark}
\label{rmk:comparison}
In the Introduction, below \cref{eq:defzeta}, we mentioned the inherent difference between the sample complexity $\sigma(A,f,\e)$ in the certified setting and the more classical sample complexity $\zeta(A,f,\e)$ in the non-certified setting.
We can now make our statements more formal.

Our paper shows that in the certified setting, the sample complexity $\sigma(A,f,\e)$ of an optimal algorithm~$A$ (e.g., $A=\cdoo$) is characterized by the quantity 
\[
    \SC(f,\e) 
:=
    \cN \lrb{ \cX_{\e}, \frac{\e}{L} } + \sum_{k=1}^{\me} \cN \lrb{ \cX_{(\e_k, \e_{k-1}]}, \, \frac{\e_k}{L} } \;.
\]
In contrast, the previous proposition shows that in the non-certified setting the sample complexity $\zeta(\ncdoo,f,\e)$ of the $\ncdoo$ algorithm is upper bounded (up to constants) by
\[
    \SNC(f,\e)
:=
    \sum_{k=1}^{\me} \cN \lrb{ \cX_{(\e_k, \e_{k-1}]}, \, \frac{\e_k}{L} } \;.
\]
The two expressions look remarkably alike but are subtly very different.
In fact, the latter depends only on the ``size'' (i.e., the packing numbers) of suboptimal points.
The former has an additional term measuring the size of near-optimal points.
Now, note that the flatter a function is, the fewer suboptimal points there are.
This implies that the sum $\sum_{k=1}^{\me} \cN \lrb{ \cX_{(\e_k, \e_{k-1}]}, \, \frac{\e_k}{L} }$ becomes very small (hence, so does $\SNC(f,\e)$), but in turn, the set of near-optimal points $\cX_\e$ becomes large (hence, so does $\SC(f,\e)$). For instance, in the extreme case of constant functions $f$, we have $\SNC(f,\e) = 0$ but $\SC(f,\e) \approx (L/\e)^d$.
This fleshes out the fundamental difference between certified and non-certified optimization, giving formal evidence to the intuition that the more ``constant'' a function is, the easier it is to recommend an $\e$-optimal point, but the harder it is to certify that such recommendation is actually a good recommendation.  
\end{remark}

\end{document}